\documentclass[11pt,psamsfonts]{amsart}
\usepackage{amsmath}
\usepackage{amsthm}
\usepackage{amssymb}
\usepackage{amscd}
\usepackage{amsfonts}
\usepackage{amsbsy}
\usepackage{graphicx}
\usepackage[dvips]{psfrag}
\usepackage{array}
\usepackage{color}
\usepackage{epsfig}
\usepackage{url}
\usepackage{ulem}

\newcolumntype{L}{>{\displaystyle}l}
\newcolumntype{C}{>{\displaystyle}c}
\newcolumntype{R}{>{\displaystyle}r}

\def \kp {\kappa}
\def \g {\gamma}
\def \a {\alpha}
\def \b {\beta}
\def \R {\mathbb{R}}

\def \E {\mathcal{E}}
\def \M {\mathcal{M}}
\def \t {\times}

\def \> {\geq}
\def \< {\leq}

\def \E {\mathcal{E}}

\def \lb {\lambda}

\def \l {\langle}
\def \r {\rangle}

\def \rt {\rightarrow}
\def \Rl {\mathcal{R}}
\def \ph {\varphi}

\newtheorem {theorem} {Theorem}
\newtheorem {proposition} [theorem] {Proposition}
\newtheorem {corollary} [theorem] {Corollary}

\newtheorem {definition} {Definition}

\newtheorem {remark} {Remark}

\begin{document}

\title[Geometric deformations of curves in the Minkowski plane]
{Geometric deformations of curves in the Minkowski plane}

\author[A.P. Francisco]
{Alex Paulo Francisco}



\maketitle

\section*{\bf Abstract}

In this paper, we extend the method developed in \cite{mostafa-tese,
mostafa} to curves in the Minkowski plane. The method proposes a way
to study deformations of plane curves taking into consideration
their geometry as well as their singularities. We deal in detail
with all local phenomena that occur generically in $2$-parameters
families of curves. In each case, we obtain the geometry of the
deformed curve, that is, information about inflections, vertices and
lightlike points. We also obtain the behavior of the evolute/caustic
of a curve at especial points and the bifurcations that can occur
when the curve is deformed.

\smallskip

\section{\bf Introduction}

%%%%%%%%%%%%%%%%%%%%%%%%%%%%%%%%%%%%%%%%%%%%%%%%%%%%%%%%%%%%%%%%%%%%%%%%%%%%%%%%%%%%%%%%%%%%%%%%%%%%%%%%%%%%%%%%%%%%%%%%
%%%%%%%%%%%%%%%%%%%%%%%%%%%%%%%%%%%%%%%%%%%%%%%%%%%%%%%%%%%%%%%%%%%%%%%%%%%%%%%%%%%%%%%%%%%%%%%%%%%%%%%%%%%%%%%%%%%%%%%%

Consider the orthogonal projections of a curve $C$ in $\R^3$ to planes.
The $\mathcal{A}$-classification of the singularities of such
projections was obtained by David in \cite{david}. The equivalence
relation by the group $\mathcal{A}$ is very important when we study
singularities of projections of a space curve. However,
diffeomorphisms destroy the geometry of the projected curve. In
\cite{nunu,raul,wall} is studied the geometry of a germ or
multi-germ of a curve together with its contact with lines.

However it is also of interest to study the geometry of curves that
are deformed. This lead to the following question \cite{mostafa}: \textit{is there
a deformations theory which takes into account the deformations of
singularity of a plane curve as well as its geometry?} This question
is still an open problem. In \cite{mostafa-tese,mostafa},
Salarinoghabi and Tari proposed a method to study such deformations
of curves in the Euclidian plane. They named such method
\textit{Flat and Round Singularity theory for plane curves} (FRS). See also \cite{wall} for an alternative approach using invariants.

In this paper, we extend the method introduced in
\cite{mostafa-tese,mostafa} to deal with curves in the Minkowski
plane. We study all the phenomena of codimension less or equal to
$2$. In section $2$ we present some preliminary concepts about the
Minkowski plane; in section $3$ we study the stable behavior of the
evolute/caustic; in section $4$ we define the concepts of
\textit{FRLS}-equivalence and \textit{FRLS}-genericity (F for flat, R for round, L for
lightlike and S for singular). In the rest of this paper we study
the geometric deformations of plane curve at a vertex of order $2$,
an inflection of order $2$ and $3$ (section $5$), a lightlike
ordinary inflection and a lightlike inflection of order $2$ (section
$6$), a non-lightlike ordinary cusp (section $7$), a lightlike
ordinary cusp (section $8$) and a non-lightlike ramphoid cusp
(section $9$).

%%%%%%%%%%%%%%%%%%%%%%%%%%%%%%%%%%%%%%%%%%%%%%%%%%%%%%%%%%%%%%%%%%%%%%%%%%%%%%%%%%%%%%%%%%
%%%%%%%%%%%%%%%%%%%%%%%%%%%%%%%%%%%%%%%%%%%%%%%%%%%%%%%%%%%%%%%%%%%%%%%%%%%%%%%%%%%%%%%%%%
\section{\bf Curves in the Minkowski plane}
%%%%%%%%%%%%%%%%%%%%%%%%%%%%%%%%%%%%%%%%%%%%%%%%%%%%%%%%%%%%%%%%%%%%%%%%%%%%%%%%%%%%%%%%%%
%%%%%%%%%%%%%%%%%%%%%%%%%%%%%%%%%%%%%%%%%%%%%%%%%%%%%%%%%%%%%%%%%%%%%%%%%%%%%%%%%%%%%%%%%%

The Minkowski plane, denoted by $\R_1^2$, is the plane $\R^2$
endowed with the metric induced by the pseudo scalar product
$\langle u,v \rangle = -u_1v_1 + u_2v_2$, where $u = (u_1,u_2)$ and
$v = (v_1,v_2)$. We say that a non-zero vector $u\in\R_1^2$ is
\textit{spacelike} if $\langle u,u \rangle > 0$, \textit{lightlike}
if $\langle u,u \rangle = 0$ and \textit{timelike} if $\langle u,u
\rangle < 0$. Given a vector $u = (u_1,u_2)\in\R_1^2$, we denote by
$u^\perp$ the vector given by $u^\perp = (u_2,u_1)$ which is
orthogonal to $u$. If $u$ is lightlike, then $u^\perp= \pm u$ and if
$u$ is spacelike (resp. timelike), then $u^\perp$ is timelike (resp.
spacelike).

Let $\gamma:I\rightarrow\R_1^2$ be a smooth and regular curve. We
say that $\gamma$ is \textit{spacelike} (resp. \textit{timelike}) if
$\gamma^\prime(t)$ is a spacelike (resp. timelike) vector for all
$t\in I$. A point $\gamma(t)$ is called a \textit{lightlike point}
if $\gamma^\prime(t)$ is a lightlike vector.

We denote by $T(t)$ the unit tangent vector of $\g$ at a
non-lightlike point $t\in I$ and consider $N(t)$ the unit normal
vector of $\g$ at $t\in I$ such that $\{T(t),N(t)\}$ is a positively
oriented basis of $\R_1^2$, that is,
$$T(t) = \frac{\g^\prime(t)}{\parallel\g^\prime(t)\parallel} \ \ \  \text{and} \ \ \ N(t) = \frac{\pm\g^\prime(t)^\perp}{\parallel\g^\prime(t)\parallel},$$
where we use $+$ if $\g(t)$ is timelike and $-$ if $\g(t)$ is
spacelike. Note that if $\g(t)$ is spacelike (resp. timelike), then
$N(t)$ is timelike (resp. spacelike).

The curvature function $\kp:I\rightarrow\R_1^2$ of $\g$ at $t$ is given by
$$\kp(t) = \frac{\langle\g^{\prime\prime}(t),\g^\prime(t)^\perp\rangle}{|\langle\g^\prime(t),\g^\prime(t)\rangle|^\frac{3}{2}}.$$

We say that a point $\g(t_0)$ is a \textit{vertex} of $\g$ if
$\kp^\prime(t_0) = 0$. Moreover, a vertex is called an
\textit{ordinary vertex} if $\kp^\prime(t_0) = 0$ and
$\kp^{\prime\prime}(t_0)\neq 0$ and a \textit{vertex of order $k$}
(with $k\geq 2$) if $\kp^\prime(t_0) = ... = \kp^{(k)}(t_0) = 0$ and
$\kp^{(k+1)}(t_0)\neq 0$. A point $\g(t_0)$ is called an
\textit{inflection} of $\g$ if $\kp(t_0) = 0$. Moreover, an
\textit{inflection} is called an \textit{ordinary inflection} if
$\kp(t_0) = 0$ and $\kp^\prime(t_0) \neq 0$ and an \textit{inflection
of order $k$} (with $k\geq 2$) if $\kp(t_0) = ... = \kp^{(k-1)}(t_0)
= 0$ and $\kp^{(k)}(t_0)\neq 0$.

The contact of $\g$ with lines orthogonal to a non-zero vector
$v\in\R_1^2$ is captured by the singularities of the height function
$h_v:I\rightarrow\R$ given by $h_v(t) = \l\g(t),v\r$. We say that a
curve $\g$ has an \textit{$A_k$-contact} (resp. \textit{$A_{\geq
k}$-contact}) with the line $l_v(t_0) = \{p\in\R_1^2;\ \l p,v\r = \l
\g(t_0),v\r\}$ at $t_0$ if $h_v$ has an $A_k$-singularity (resp.
$A_{\geq k}$-singularity) at $t_0$.

Using the contact of $\g$ with its tangent line, we can define an
inflection in a general way (including at lightlike points) as
follows. A point $t_0$ is called an \textit{inflection} if $\g$ has
an $A_{\geq 2}$-contact with its tangent line at $t_0$. Moreover, an
inflection is called an \textit{ordinary inflection} (resp.
\textit{lightlike ordinary inflection}) if the contact is of order
$2$ and $t_0$ is not (resp. is) a lightlike point and an
\textit{inflection of order $k$} (resp. \textit{lightlike inflection
of order $k$}), with $k\geq 2$, if the contact is of order $k+1$ and
$t_0$ is not (resp. is) a lightlike point. The two definitions are
equivalents at non-lightlike points.

The \textit{evolute} of $\g$ at a non-lightlike point and away from
inflections is the curve in $\R_1^2$ parametrized by
$$e(t) = \g(t) - \frac{1}{\kp(t)}N(t).$$

\begin{proposition}\label{prop:evoluta-fora}
Let $\g:I\rt\R_1^2$ be a smooth and regular curve and let $t_0\in I$
be a point which is neither a lightlike point nor an inflection point.
Then, the evolute of $\g$ lies locally in the half-plane
determined by the tangent line of $\g$ at $t_0$ which does not
contain the curve $\g$.
\end{proposition}
\begin{proof} It follows similarly to that of curves in the Euclidian case (\cite{Keti}, p. 47).
\end{proof}

The family of distance squared function on a smooth and regular curve
$\g:I\rt\R_1^2$ is the family of maps $d:I\t\R_1^2\rt\R$ given by
$d(t,u) = \l\g(t)-u,\g(t)-u\r$. The \textit{caustic} of $\g$ is the
bifurcation set of the family $d$,
that is
$$Bif(d) = \{u\in\R_1^2;\ \exists\ t\in I \ \text{ such \ that} \ d_u^\prime(t) = d_u^{\prime\prime}(t) = 0\}$$
and is a locus of centers of pseudo circles which have an $A_{\geq
2}$-contact with $\g$. For more details about caustics and its
singularities see, for example,
\cite{Arnold72,Arnold,Saloom-tese,Minkowski,zak76}.

We have
$$Bif(d) = \{\g(t) - \lambda\g^\prime(t)^\perp; \ \text{such that} \  \lambda\l\g^\prime(t)^\perp,\g^{\prime\prime}(t)\r +
\l\g^\prime(t),\g^\prime(t)\r = 0,  \ t\in I\ ,  \lambda\in\R \}.$$ 

The caustic of $\g$
is well defined at lightlike and singular points and coincide with
the evolute away from such points (\cite{Fukunaga-takahashi, Minkowski}).

%%%%%%%%%%%%%%%%%%%%%%%%%%%%%%%%%%%%%%%%%%%%%%%%%%%%%%%%%%%%%%%%%%%%%%%%%%%%%%%%%%%%%%%%%%
%%%%%%%%%%%%%%%%%%%%%%%%%%%%%%%%%%%%%%%%%%%%%%%%%%%%%%%%%%%%%%%%%%%%%%%%%%%%%%%%%%%%%%%%%%
\section{\bf Stable configurations of the caustic}
%%%%%%%%%%%%%%%%%%%%%%%%%%%%%%%%%%%%%%%%%%%%%%%%%%%%%%%%%%%%%%%%%%%%%%%%%%%%%%%%%%%%%%%%%%
%%%%%%%%%%%%%%%%%%%%%%%%%%%%%%%%%%%%%%%%%%%%%%%%%%%%%%%%%%%%%%%%%%%%%%%%%%%%%%%%%%%%%%%%%%

The stable phenomena on curves in the Minkowski plane are lightlike points, ordinary vertices and
inflections. In this section we describe the configurations of the
caustic at such points. All figures are draw in blue and red, blue
for the curve and red for its caustic. Moreover, we represent
lightlike points, vertices and inflections, respectively, by stars,
discs and squares.

We start with lightlike points. For this we consider $\Omega$ the
set of smooth and regular curves $\g:S^1\rt \R_1^2$ which satisfies
$\langle\gamma^{\prime\prime}(t),\gamma^\prime(t)\rangle \neq 0$
when $\langle\gamma^\prime(t),\gamma^\prime(t)\rangle = 0$, that is,
their lightlike points are not lightlike inflections.

\begin{proposition} {\rm \text{(\cite{Minkowski})}}
Let $\g\in\Omega$. Then the caustic of $\g$ is a regular curve at
the lightlike point of $\g$ and it has an ordinary tangency with
$\g$ at such points. Moreover, $\g$ and its caustic lie locally
at opposite sides of the common tangent line at the lightlike point.
\end{proposition}

The following result gives the configuration of the evolute at an
ordinary vertex.

\begin{proposition}\label{prop:vert-ord}
Let $\gamma:I\rightarrow\R_1^2$ be a smooth and regular curve.
Suppose $t_0\in I$ is an ordinary vertex of $\gamma$ which is not an
inflection. Then, locally at $t_0$, the configuration of $\gamma$
and its evolute is as in Figure \ref{fig:vert.ord} left if
$\kp(t_0)\kp^{\prime\prime}(t_0)<0$ and right if
$\kp(t_0)\kp^{\prime\prime}(t_0)>0$.
\end{proposition}

\begin{figure}[h!]
\centering
\includegraphics[scale = 0.9]{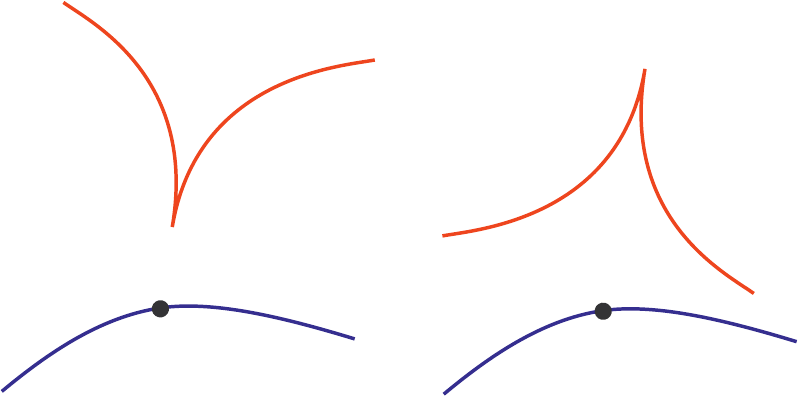}
\caption{\small{Behavior of the evolute at an ordinary
vertex.}}\label{fig:vert.ord}
\end{figure}

\begin{proof}
We take $\gamma$ parametrized by arc-length. Differentiating the
parametrization $e(t) = \gamma(t) - \frac{1}{\kp(t)}N(t)$ of the
evolute three times and evaluating at $t_0$ we get\linebreak
$\displaystyle e^{\prime\prime}(t_0) =
\frac{\kp^{\prime\prime}(t_0)}{(\kp(t_0))^2}N(t_0)$ and $\displaystyle
e^{\prime\prime\prime}(t_0) =
\frac{2\kp^{\prime\prime}(t_0)}{\kp(t_0)}T(t_0) +
\frac{\kp^{\prime\prime\prime}(t_0)}{(\kp(t_0))^2}N(t_0)$. Hence, in the coordinate system with basis $\{T(t_0),N(t_0)\}$, we
can write the $3$-jet of $e(t)-e(t_0)$ at $t_0$ in the form
$$j^3(e(t) - e(t_0)) = (\frac{\kp^{\prime\prime}(t_0)}{3\kp(t_0)}(t-t_0)^3,\frac{\kp^{\prime\prime}(t_0)}{2(\kp(t_0))^2}(t-t_0)^2 + \frac{\kp^{\prime\prime\prime}(t_0)}{6(\kp(t_0))^2}(t-t_0)^3).$$

We have two possibilities at the cusp of the evolute: turning
towards the curve or away from it (see Figure \ref{fig:vert.ord}).
The cusp of the evolute turns towards (resp. away from) the curve
when $\kp(t_0)\kp^{\prime\prime}(t_0) > 0$ (resp.
$\kp(t_0)\kp^{\prime\prime}(t_0) < 0$).

%Note that if $\kp^{\prime\prime}(t_0)
%> 0$ the cusp at the evolute is contained in the half-plane, determined by the tangent line at $t_0$, which
%contains $N(t_0)$ and if $\kp^{\prime\prime}(t_0) < 0$ it is
%contained in the half-plane which does not contain $N(t_0)$.
%
%We know that $T^\prime(t) = \kp(t)N(t)$, that is,
%$\g^{\prime\prime}(t) = \kp(t)N(t)$. Thus, if $\kp(t_0) > 0$ then
%$\g^{\prime\prime}(t_0)$ and $N(t_0)$ have the same direction, and
%if $\kp(t_0) < 0$ then $\g^{\prime\prime}(t_0)$ and $N(t_0)$ have
%opposite directions. Analyzing each case we obtain the result.
\end{proof}

\begin{definition}
An ordinary vertex $t_0\in I$ is called an \textit{inward vertex} if
$\kp(t_0)\kp^{\prime\prime}(t_0)>0$ and an \textit{outward vertex} if
$\kp(t_0)\kp^{\prime\prime}(t_0)<0$.
\end{definition}

The next result gives the configuration of the evolute at an
ordinary inflection.

\begin{proposition}\label{prop:inf-ord}
Let $\gamma:I\rightarrow\R_1^2$ be a smooth and regular curve.
Suppose that $\gamma$ has an ordinary inflection at $t_0$. Then the
evolute of $\gamma$ tends to infinity asymptotically along the
normal to $\gamma$ at $t_0$ and can be modeled by $xy = 1$ in some
coordinate system. The two components of the evolute lie in the two
quadrants determined by the tangent and normal lines of $\gamma$ at
$t_0$ which does not contain the curve $\gamma$, see Figure
\ref{fig:inf.ord}.
\end{proposition}

\begin{figure}[h!]
\centering
\includegraphics[scale = 0.8]{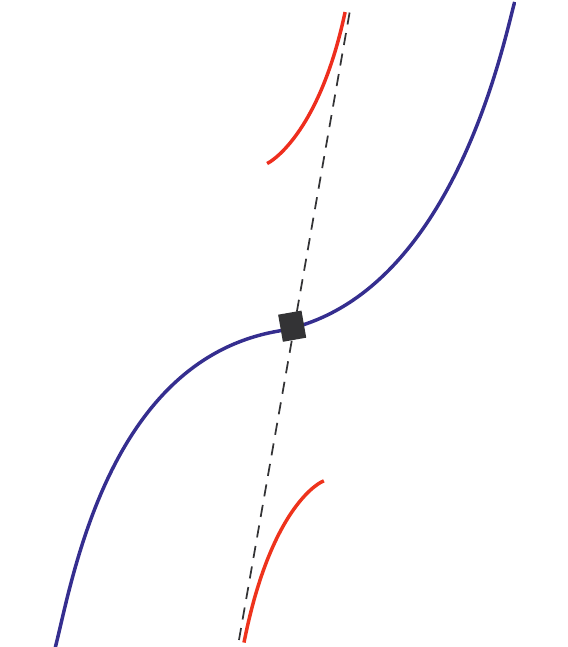}
\caption{\small{Evolute at an ordinary
inflection.}}\label{fig:inf.ord}
\end{figure}

\begin{proof}
We consider the case when $\g$ is a timelike curve, the spacelike case is analogous and is omitted. We
can suppose that $t_0 = 0$ and $\g(t) = (t,f(t)) = (t,c_2t^2 +
c_3t^3 + O(t^4))$. As $t_0 = 0$ is an ordinary inflection of $\g$,
we have $\kp(0) = 0$ and $\kp^\prime(0)\neq 0$, that is, $c_2 = 0$ and
$c_3\neq 0$. Thus, the expression of the evolute is
$$e(t) = (\frac{1}{2}t + O(t^2),\frac{1}{6c_3t}(-1 + O(t))),$$
and the results follows.
\end{proof}

%%%%%%%%%%%%%%%%%%%%%%%%%%%%%%%%%%%%%%%%%%%%%%%%%%%%%%%%%%%%%%%%%%%%%%%%%%%%%%%%%%%%%%%%%%%%%%%%%%%%%%%%%%%%%%%%%%%%%%%%
%%%%%%%%%%%%%%%%%%%%%%%%%%%%%%%%%%%%%%%%%%%%%%%%%%%%%%%%%%%%%%%%%%%%%%%%%%%%%%%%%%%%%%%%%%%%%%%%%%%%%%%%%%%%%%%%%%%%%%%%
\section{\bf \textit{FRLS}-equivalence: general concepts}
%%%%%%%%%%%%%%%%%%%%%%%%%%%%%%%%%%%%%%%%%%%%%%%%%%%%%%%%%%%%%%%%%%%%%%%%%%%%%%%%%%%%%%%%%%%%%%%%%%%%%%%%%%%%%%%%%%%%%%%%
%%%%%%%%%%%%%%%%%%%%%%%%%%%%%%%%%%%%%%%%%%%%%%%%%%%%%%%%%%%%%%%%%%%%%%%%%%%%%%%%%%%%%%%%%%%%%%%%%%%%%%%%%%%%%%%%%%%%%%%%

We present a method for studying the geometry of deformations of
plane curves in the Minkowski plane, which we denote by
\textit{FRLS-deformations} (F for Flat, R for Round, L for Lightlike
and S for Singular). The method is an extension of that in
\cite{mostafa}; here we need to consider, additionally, the lightlike points.

\begin{definition}\label{def:FRLS-equivalent} {\rm (Compare with Definition 1.1 in \cite{mostafa})}
Consider two germs of $m$-parameter deformations $\g_s$ and
$\alpha_u$ of a curve in $\R_1^2$ and consider the parameter space
of $\g_s$ endowed with a stratification $S_1$ such that if $s_1$ and
$s_2$ belong to the same stratum, then $\g_{s_1}$ and $\g_{s_2}$
\begin{itemize}
\item[\text{(i)}] are diffeomorphics;

\item[\text{ (ii)}] have the same number of inflections, vertices and lightlike points;

\item[\text{(iii)}] have the same relative position of their singularities, self-intersection points, inflections, vertices and lightlike points.
\end{itemize}
Consider also the parameter space of $\a_u$ endowed with another
stratification $S_2$ satisfying properties {\rm(i)-(iii)}. We say
that $\g_s$ and $\a_u$ are \textit{FRLS-equivalents} if there is a
germ of homeomorphism $k:\R^m,(S_1,0)\rightarrow\R^m,(S_2,0)$ such
that $\g_s$ and $\a_{k(s)}$ satisfies {\rm (i)-(iii)}.
\end{definition}

We can define the notion of \textit{FRLS}-equivalence of deformations of a
plane curve with different number of parameters. We say that an
$m$-parameter deformation $\g_s$ and an $n$-parameter deformation
$\alpha_u$ (say $n\geq m$) of a plane curve are \textit{FRLS}-equivalent if
$\widetilde{\g}_s$ and $\alpha_u$ are \textit{FRLS}-equivalents in the above
sense, where $\widetilde{\g}_s$ is an $n$-parameter deformation
given by
$\widetilde{\g}_s(t,s_1,...,s_m,\widetilde{s}_{m+1},...,\widetilde{s}_n)
= \g_s(t,s_1,...,s_m)$.

In each case that we deal with, the stratification of the parameter
space is obtained from a stratification in the jet or multi-jet
space. The stratification of the jet or multi-jet space is given by
equations which determine each phenomenon. The notation that we use
for local strata are $I(k), V(k), LI(k)$ for inflections, vertices
and lightlike inflections of order $k$, respectively, (or $I, V, LI$
for the ordinary case), $L$ for lightlike points, $C$ for cusps,
$LC$ for lightlike cusps and $RC$ for ramphoid cusps. The
multi-local strata are denoted by $IT, VT, LT$ for points of transverse
intersection between two branches where one of them is an
inflection, a vertex, a lightlike point, respectively, and $Tc$
(Tacnode) for points of tangential intersection between two
branches. To obtain such stratification, we use the Monge-Taylor
map, which we define as follows.

Let $\g(t) = (\a(t),\b(t))$ be a plane curve. At each point $t_0$ in
the source of $\g$ we write
\begin{eqnarray*}
j^k_{t_0}\g(t) &=& (\a(t_0)+\a^\prime(t_0)(t-t_0) +
\frac{1}{2!}\a^{\prime\prime}(t_0)(t-t_0)^2 \dots +
\frac{1}{k!}\a^{(k)}(t_0)(t-t_0)^k;\\ & &
\b(t_0)+\b^\prime(t_0)(t-t_0) +
\frac{1}{2!}\b^{\prime\prime}(t_0)(t-t_0)^2 \dots
+\frac{1}{k!}\b^{(k)}(t_0)(t-t_0)^k).
\end{eqnarray*}

We take $\g(t_0) = (\a(t_0),\b(t_0)) = (0,0)$. We define the
\textit{Monge-Taylor map} as the map $j^k\phi_\g:\R,0\rt J^k(1,2)$
given by $j^k\phi_\g(t) = (j^k_t \a,j^k_t\b)$. Identifying
$J^k(1,2)$ with $\R^k\t\R^k$, we write
\begin{equation}\label{Monge-Taylor:sing}
j^k\phi_\g(t) =
(\a^\prime(t),\frac{1}{2!}\a^{\prime\prime}(t),\dots,
\frac{1}{k!}\a^{(k)}(t); \b^\prime(t),
\frac{1}{2!}\b^{\prime\prime}(t),\dots, \frac{1}{k!}\b^{(k)}(t)).
\end{equation}

We can simplify the map $j^k\phi_\g$ depending on the case in consideration:
\begin{itemize}
\item[(1)] If $t = 0$ is a regular but not a lightlike point of $\g$, then we
take $\g(t) = (t,\b(t))$, with $\b^\prime(0) = 0$, for the
timelike case and $\g(t) = (\a(t),t)$, with $\a^\prime(0) = 0$, for
the spacelike case. The Monge-Taylor is taken as the map $j^k\phi_\g:\R,0\rt J^k(1,1)$ given by
\begin{eqnarray}\label{Monge-Taylor:time}
j^k\phi_\g(t) &=& (\frac{1}{2!}\b^{\prime\prime}(t),\dots, \frac{1}{k!}\b^{(k)}(t)), \ \ \text{if} \ \g \ \text{is \ timelike}\\
j^k\phi_\g(t) &=& (\frac{1}{2!}\a^{\prime\prime}(t),\dots,
\frac{1}{k!}\a^{(k)}(t)), \ \ \text{if} \ \g \ \text{is \
spacelike}\nonumber
\end{eqnarray}

\item[(2)] If $t = 0$ is a lightlike regular point of $\g$ we can take $\g(t) = (t,\b(t))$, with
$\b^\prime(0) = \pm 1$. In this case, the Monge-Taylor map is taken as the map \linebreak $j^k\phi_\g:\R,0\rt J^k(1,1)$ given by
\begin{equation}\label{Monge-Taylor:light}
j^k\phi_\g(t) =
(\b^\prime(t),\frac{1}{2!}\b^{\prime\prime}(t),\dots,
\frac{1}{k!}\b^{(k)}(t)).
\end{equation}

\item[(3)] If $\g$ is singular at $t = 0$, then the Monge-Taylor map is taken as in (\ref{Monge-Taylor:sing}).
\end{itemize}

The stratification in the space $ _2J^k(1,2)\subset
J^k(1,2)\t J^k(1,2)$ of bi-jets (see \cite{Livro-Farid}, p. 47) is used when
$\g_s$ has self-intersection points. We define
the \textit{Monge-Taylor bi-jet map} as the map $
_2j^k\phi_\g:\R\t\R,(0,0)\rt\,  _2J^k(1,2)$ given by
\begin{eqnarray}\label{Monge-Taylor:multi}
_2j^k\phi_\g(t_1,t_2) &=& \big((\a(t_1),\a^\prime(t_1),\dots,\frac{1}{k!}\a^{(k)}(t_1); \b(t_1),\b^\prime(t_1),\dots,\frac{1}{k!}\b^{(k)}(t_1)),\nonumber\\
& & (\a(t_2),\a^\prime(t_2),\dots,\frac{1}{k!}\a^{(k)}(t_2);
\b(t_2),\b^\prime(t_2),\dots,\frac{1}{k!}\b^{(k)}(t_2))\big).
\end{eqnarray}
We can define in the analogous way the Monge-Taylor n-jet map, for
$n\> 3$.

\begin{definition}\label{def:FRLS-def.}
We say that a germ of an $m$-parameter deformation $\g_s$ of a plane
curve $\g = \g_0$ is an \textit{FRLS-generic} deformation if the
family of Monge-Taylor maps induced by $\g_s$ is transverse to the
stratification of the jet space determined by the $\mathcal{A}$-invariant strata and the geometric strata $I(k), V(k), LI(k)$, and their intersections.
\end{definition}

\begin{theorem}\label{teo:IL-V}
In $J^n(1,2)$, the stratum of lightlike inflections of order $k$ is
contained in the stratum of vertices of order $2k$, that is,
$LI(k)\subset V(2k)$, for any integer $k\> 1$. Moreover, $2k$ is the
largest integer with such property, that is, $LI(k)\subset
V(2k)\backslash V(2k+1)$.
\end{theorem}
\begin{proof}
Let $\g(t) = (t,f(t))$ with a lightlike inflection of order $k$ at
$t_0 = 0$. Consider $\g_s$ an $m$-parameter deformation of $\g =
\g_0$ which we can write in the form $\g_s(t) = (t,f_s(t))$, and let
$g_s(t) = f_s^{\prime\prime\prime}(t)(1-f_s^\prime(t)^2) +
3f_s^\prime(t)f_s^{\prime\prime}(t)^2$ be the numerator of
$\kp_s^\prime$. Then, 
\begin{eqnarray*}
LI(k): && a_1 = 1, a_2 = \dots = a_{k+1}
= 0\\
V(2k): && g_s = g_s^\prime = \dots =
g_s^{(2k-1)} = 0
\end{eqnarray*}
with $a_i = f_s^{(i)}(0)/(i!)$.

The proof follows by induction on $k\> 1$. For $k = 1$, we have
$LI\subset V(2)$.

Suppose $LI(k)\subset V(2k)$. Considering $a_1 =1$ and $a_2 = \dots
= a_{k+2} = 0$, then clearly we have $g_s = \dots = g_s^{(2k-1)} = 0$.
To show that $g_s^{(2k)} = g_s^{(2k+1)} =
0$ we use the following formula obtained differentiating $g_s$
successively
\begin{eqnarray}\label{eq:derivada-gs} g_s^{(n)} &=&
(n+3)!a_{n+3}\nonumber\\ && -\sum_{i=0}^n{n \choose
i}(i+3)!a_{i+3}(\sum_{l=0}^{n-i}{n-i\choose
l}(n-i-l+1)!(l+1)!a_{n-i-l+1}a_{l+1})\\
&& + 3\sum_{j=0}^n {n\choose
j}(n-j+1)!a_{n-j+1}(\sum_{l=0}^j{j\choose
l}(j-l+2)!(l+2)!a_{j-l+2}a_{l+2}).\nonumber
\end{eqnarray}

For $n=2k$ in (\ref{eq:derivada-gs}) we get $g_s^{(2k)} =
(2k+3)!a_{2k+3} - (2k+3)!a_{2k+3} = 0$. For $n = 2k+1$ we get
$g_s^{(2k+1)} = (2k+4)!a_{2k+4} - (2k+4)!a_{2k+4} = 0$. Therefore,
$LI(k+1)\subset V(2(k+1))$.

To prove that $LI(k)\subset V(2k)\backslash
V(2k+1)$, as $LI(k)\subset V(2k)$, it is enough to show that
$g_s^{(2k)}\neq 0$ when $a_1 = 1, a_2 = \dots = a_{k+1} = 0$ and
$a_{k+2}\neq 0$. Setting $n = 2k$ in (\ref{eq:derivada-gs}) we get
$g_s^{(2k)} = (2k)!(k+2)^2(k+1)(k+3)a_{k+2}^2$. As $a_{k+2}\neq 0$, we have
$g_s^{(2k)}\neq 0$.
\end{proof}

%%%%%%%%%%%%%%%%%%%%%%%%%%%%%%%%%%%%%%%%%%%%%%%%%%%%%%%%%%%%%%%%%%%%%%%%%%%%%%%%%%%%%%%%%%%%%%%%%%%%%%%%%%%%%%%%%%%%%%%%
%%%%%%%%%%%%%%%%%%%%%%%%%%%%%%%%%%%%%%%%%%%%%%%%%%%%%%%%%%%%%%%%%%%%%%%%%%%%%%%%%%%%%%%%%%%%%%%%%%%%%%%%%%%%%%%%%%%%%%%%
\section{\bf Geometric deformations of regular curves at non-lightlike points}
%%%%%%%%%%%%%%%%%%%%%%%%%%%%%%%%%%%%%%%%%%%%%%%%%%%%%%%%%%%%%%%%%%%%%%%%%%%%%%%%%%%%%%%%%%%%%%%%%%%%%%%%%%%%%%%%%%%%%%%%
%%%%%%%%%%%%%%%%%%%%%%%%%%%%%%%%%%%%%%%%%%%%%%%%%%%%%%%%%%%%%%%%%%%%%%%%%%%%%%%%%%%%%%%%%%%%%%%%%%%%%%%%%%%%%%%%%%%%%%%%

Here, the Monge-Taylor map does not intersect the strata involving singularities and lightlike points, so an \textit{FRLS}-generic family (see Definition \ref{def:FRLS-def.}) will be called \textit{FR}-generic (we drop the letters S and L). We study \textit{FR}-generic deformations of
a plane curve at an inflection of finite order and the
bifurcations in the curve and in its evolute at an inflection of
order $2$ and $3$.

Away from inflections, the deformations of a regular curve at a
vertex of finite order can be studied using the family of distance
squared functions. If the family is an $\mathcal{R}^+$-versal
deformation, we get the deformations of the evolute whose
singularities gives the vertices of the curve (see
\cite{Arnold,generic-geometry}).

The deformations of a regular curve at an inflection of finite order
can be studied using the family of height functions. If the family
is an $\mathcal{R}$-versal deformation, then we have a model of its
discriminant and consequently a model for the
dual curve whose singularities give the inflections (see
\cite{generic-geometry}). However, if the inflection is of order
$k\geq 2$, then $\kp = \kp^\prime = 0$ at that point, so there is a
vertex concentrated at the inflection. In this way, when we deform
an inflection of order $\> 2$ we also need to consider how the
vertices appear on the deformed curve. Before that, we need the following result.

\begin{proposition}\label{prop:vert-ord2}
Let $\g$ be a smooth and regular curve. Suppose that $t_0$ is a
vertex of order $2$ which is not an inflection. Then the evolute of
$\g$ has an $E_6$ singularity at $t_0$ (see Figure
\ref{fig:vert.ord2} center). Moreover, if $\g_s$ is a generic
$1$-parameter family with $\g_0 = \g$, then the evolute of $\g_s$
undergoes the swallowtail transitions as shown in Figure
\ref{fig:vert.ord2}. We have an inward and an outward vertices on
one side of the transition and none on the other.
\end{proposition}

\begin{figure}[h!]
\centering
\includegraphics[scale = 0.75]{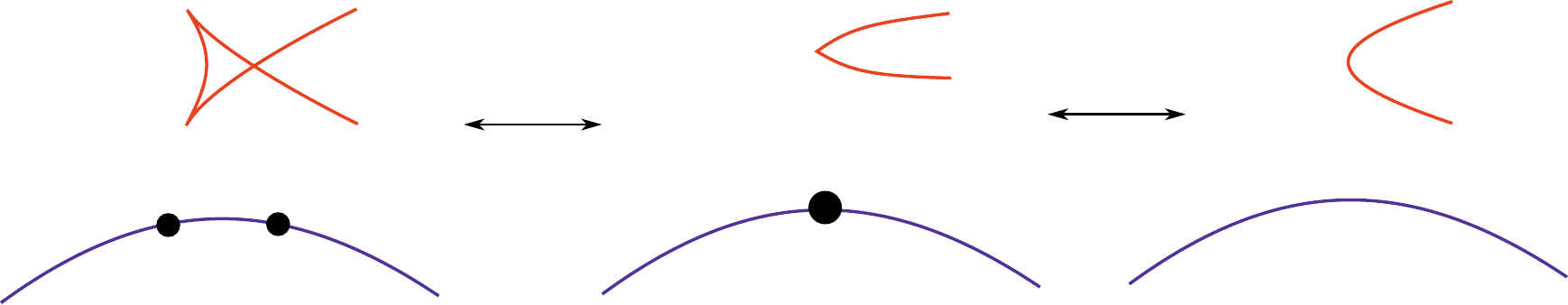}
\caption{\small{Bifurcations of the evolute at a vertex of order
2.}}\label{fig:vert.ord2}
\end{figure}

\begin{proof}
It is shown in \cite{reeve} that the evolute of $\g$ has an $E_6$
singularity at $t_0$, that is, $\g$ is $\mathcal{A}$-equivalent to
$(t^3,t^4)$. As $\g_s$ is a generic family, it follows that the
family of distance squared function $d$ of $\g_s$ is an
$3$-parameter $\mathcal{R}^+$-versal deformation of the $A_4$
singularity of $d_0$. Hence the bifurcation set of the family is a swallowtail and the sections are generic. Consequently, the evolute of $\g_s$
corresponds to the generic sections of the swallowtail (see
\cite{Bruce}).
\end{proof}

The next result gives conditions for a deformation to be \textit{FR}-generic at an inflection of order $k$.

\begin{theorem}\label{prop:versalidade-inf}
Let $\g_s$ be an $m$-parameter deformation of a regular curve $\g_0
= \g$ which has an inflection of order $n\> 2$ at $t_0$. Suppose
that $t_0$ is not a lightlike point. Then, the family of height
functions on $\g_s$ is an $\mathcal{R}^+$-versal deformation of the
height function on $\g_0$ at $t_0$ if, and only if, the family of
Monge-Taylor maps $j^k\Phi$, $k\> n+1$, is transverse to the strata
$I, I(2),\dots,I(n), V, V(2), \dots, V(n-1)$ at $j^k\phi_{\g_0}(0)$.
\end{theorem}
\begin{proof}
We suppose $\g$ timelike at $t_0 = 0$ (the spacelike case is
analogous) and take $\g = \g_0$ and $\g_s$, respectively, in the
form $\g(t) = (t,f(t))$ and $\g_s(t) = (t,f_s(t))$, with $f_s(0) =
f_s^\prime(0) = 0$, $f^{\prime\prime}(0) = \cdots = f^{(n+1)}(0) =
0$, $f^{(n+2)}(0) = (n+2)!c_{n+2}\neq 0$. The stratum of
inflections of order $n$ is given by $I(n): a_2 = \dots = a_{n+1} =
0$ and the stratum of vertices of order $n$ is given by $V(n): g_s =
\dots = g_s^{(n-1)} = 0$, where $g_s(t) =
f_s^{\prime\prime\prime}(t)(1-f_s^\prime(t)^2) +
3f_s^\prime(t)f_s^{\prime\prime}(t)^2$.

It is clear that $I(n)\subset V(n-1)$. Hence, its enough to verify
the transversality of the family of Monge-Taylor maps to
the stratum $I(n)$. Note that the
tangent space of the stratum $I(n)$ in $J^k(1,1)$ is
$\R\cdot\{e_1,e_{n+2}, \dots, e_{k}\}$, where $e_i$ is the $i$-th
vector of the canonical basis of $\R^k$. Therefore, $j^k\Phi$ is
transversal to the stratum $I(n)$ if, and only if,
$e_2,\dots,e_{n+1} \in Im(dj^k\Phi(0,0)) + \R\cdot\{e_1, e_{n+2},
\dots, e_k\}$.

Note that $j^k\Phi(t,s) =
(f_s^\prime(t),\frac{1}{2!}f_s^{\prime\prime}(t),\dots,\frac{1}{k!}f_s^{(k)}(t))$
and consequently the jacobian matrix of $j^k\Phi$ at $(0,0)$ is
$$\left(\begin{array}{cc}
    0 &  \frac{1}{j!}\frac{\partial^{j+1} f}{\partial t^j \partial s_i}(0,0)\\
    (l+1)c_{l+1} &  \frac{1}{l!}\frac{\partial^{l+1} f}{\partial t^l \partial s_i}(0,0)
\end{array}\right)_{k\t (m+1)},$$
with $j = 1,\dots, n$, $l = n+1, \dots, k$ and $i = 1,\dots, m$.
Taking $u = (\frac{1}{(n+2)c_{n+2}},0,\dots,0)\in\R^{m+1}$ we have
$dj^k\Phi(0,0)(u) = e_{n+1} + v$, where
$v\in\R\cdot\{e_{n+2},\dots,e_k\}$. Thus, $e_2,\dots, e_n\in
Im(dj^k\Phi(0,0)) + \R\cdot\{e_1, e_{n+2}, \dots, e_k\}$ if, and
only if, the matrix $(\frac{1}{j!}\frac{\partial^{j+1} f}{\partial
t^j \partial s_i}(0,0))_{(n-1)\t m}$ has rank $n-1$, which is
exactly the condition for $\mathcal{R}^+$-versality of the family of
height function on $\g_s$.
\end{proof}

\begin{corollary}\label{def:generic-inflexão}
A deformation $\g_s$ of a regular curve $\g$ at an inflection of
finite order is \textit{FR}-generic if, and only if, the family of height
functions on $\g_s$ is an $\mathcal{R}^+$-versal deformation of the
singularity of the height function on $\g$ along the normal
direction.
\end{corollary}

\begin{remark}\label{teo:nlightlike}
Denote by $H_{(u,s)}$, $\kp_s$ and $\kp_s^\prime$, the family of
height functions on $g_s$, the curvature function and the derivatives of the
curvature function of $\g_s$, respectively. By direct computations
we can show, in the hypothesis of the Theorem
\ref{prop:versalidade-inf}, that
\begin{itemize}
\item[(i)] $H_{(u,s)}$ is an
$\mathcal{R}^+$-versal deformation of $H_{(u_0,0)}$ if, and only if,
$\kp_s$ is an $\mathcal{R}$-versal deformation of the $\kp_0$.

\item[(ii)] If $\kp_s$ is an $\mathcal{R}$-versal deformation of $\kp_0$, then $\kp_s^\prime$ is an $\mathcal{R}$-versal deformation of $\kp_0^\prime$.
\end{itemize}
\end{remark}

In the rest of this section we study the \textit{FR}-bifurcations of the
curve and its evolute at an inflection of order $2$ and $3$.

\begin{theorem}\label{teo:I(2)}
Let $\g$ be a regular curve with an inflection of order 2 at $t_0$,
with $t_0$ a non-lightlike point.
\begin{itemize}
\item[\text{ (i)}] The evolute goes to infinity asymptotically along
the normal line of $\g$ at $t_0$ and can be modeled by $x^2y = 1$
in some coordinate system. The two components of the evolute lie in
the two quadrants determined by the tangent and normal lines of $\g$
at $t_0$ which do not contain the curve $\g$ (see Figure
\ref{fig:inf.ord2} center).

\item[\text{ (ii)}] Let $\g_s$ be an \textit{FR}-generic $1$-parameter family
with $\g_0 = \g$. Then the bifurcations in the evolute are as shown
in Figure \ref{fig:inf.ord2}. We have two ordinary inflections and
an outward vertex on one side of the transition and an inward vertex
and no inflections on the other side of the transition.

\item[\text{ (iii)}] Any \textit{FR}-generic $1$-parameter family of $\g$ is \textit{FR}-equivalent to the model family
$\a_u(t) = (t,t^4 + ut^2)$.
\end{itemize}
\end{theorem}

\begin{figure}[h!]
\centering
\includegraphics[scale = 1]{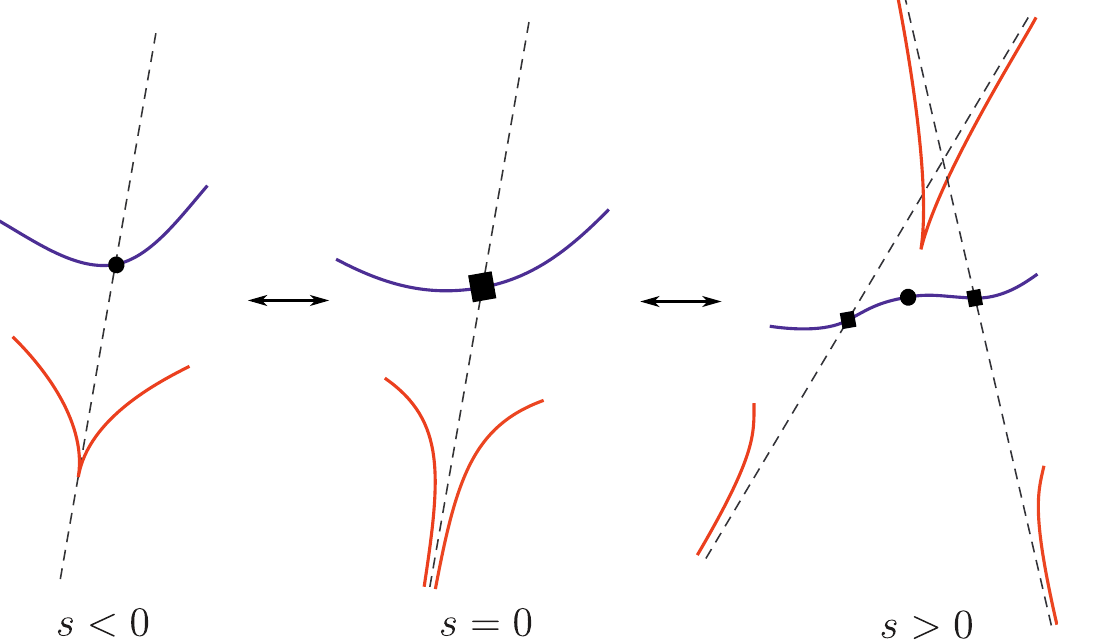}
\caption{\small{\textit{FR}-generic deformations of a curve and its evolute at an inflection of order
2.}}\label{fig:inf.ord2}
\end{figure}

\begin{proof}
We deal with the case $\g$ a timelike curve, the spacelike case follows analogously.

(i) As $t_0$ is an inflection of order $2$ of $\g$, we take $t_0
= 0$ and $\g(t) = (t,f(t)) = (t,c_4t^4 + O(t^5))$, with $c_4\neq 0$.
We have $e(t) = (\frac{2}{3}t + O(t^2), -\frac{1}{12c_4t^2}(1 +
O(t)))$, which can be modeled by $x^2y = 1$ in some coordinate
system.

(ii) As $\g_s$ is \textit{FR}-generic, we can take $\g_s(t) = (t,f_s(t)) =
(t, st^2 + \overline{c_3}(s)t^3 + (c_4 + \overline{c_4}(s))t^4 +
O_s(t^5))$, where $\overline{c_i}(0) = 0$, with $i = 3,4$, (see
Corollary \ref{def:generic-inflexão}).

The family $\kp_s$ is an $\mathcal{R}$-versal deformation of the
$A_1$ singularity of $\kp_0$ (Remark \ref{teo:nlightlike}).
Therefore, $\kp_s$ is equivalent to $t^2+u$. Then on one side of the
transition $\kp_s$ has a critical point and no zeros, while on the
other side it has one critical point between two zeros.

At critical points of $\kp_s$ on both sides of the transition the
sign of $\frac{\partial^2\kp}{\partial t^2}(t,s)$ is the same, but
the sign of $\kp_s$ changes. Therefore, on the side of the transition
which does not have inflections, the vertex of $\g_s$ is inward and
on the other side it is outward.

The bifurcations of $\g_s$ and of its evolute are as shown in Figure
\ref{fig:inf.ord2} (where we also used Proposition \ref{prop:vert-ord}
and \ref{prop:inf-ord}).

(iii) The calculations in (ii) depend only on the fact that the
curve has an inflection of order $2$ and on the family being an
\textit{FR}-generic deformation. The family $\a_u$ satisfies these conditions
and can be taken as an \textit{FR}-model.
\end{proof}

\begin{theorem}
Let $\g$ be a regular curve with an inflection of order 3 at $t_0$.
\begin{itemize}
\item[\text{ (i)}] The evolute goes to infinity asymptotically along
the normal line of $\g$ at $t_0$ and can be modeled by $x^3y = 1$
in some coordinate system. The two components of the evolute lie in
the two quadrants determined by the tangent and normal lines of $\g$
at $t_0$ which do not contain the curve $\g$ (see Figure
\ref{fig:inf.ord3}).

\item[\text{ (ii)}] Let $\g_s$ be an \textit{FR}-generic $2$-parameter family with $\g_0 = \g$. Then the bifurcations in the evolute are as shown in Figure
\ref{fig:inf.ord3}.

\item[\text{ (iii)}] Any \textit{FR}-generic $2$-parameter family of $\g$ is \textit{FR}-equivalent to the model family
$\a_u(t) = (t,t^5 + u_1t^2 + u_2t^3)$.
\end{itemize}
\end{theorem}

\begin{proof}
We deal with the case $\g$ a timelike curve, the spacelike case follows analogously.

(i) As $t_0$ is an inflection of order $3$ of $\g$, we can take
$t_0 = 0$ and $\g(t) = (t,f(t)) = (t,c_5t^5 + O(t^6))$, with
$c_5\neq 0$. We have $e(t) = \displaystyle (\frac{3}{4}t + O(t^2),
-\frac{1}{20c_5t^3}(1 + O(t)))$, which can be modeled by $x^3y = 1$
in some coordinate system, see Figure \ref{fig:inf.ord3}
$\textcircled{1}$.

(ii) As $\g_s$ is \textit{FR}-generic, we can take $\g_s(t) = (t,f_s(t)) =
(t,s_1t^2 + s_2t^3 + \overline{c_4}(s)t^4 + (c_5 +
\overline{c_5}(s))t^5 + O_s(t^6))$, where $\overline{c_i}(0) = 0$,
for $i = 4,5$ and $s = (s_1,s_2)$, (see Corollary
\ref{def:generic-inflexão}).

The strata $I(2)$ and $V(2)$ are, respectively, the discriminant and
the bifurcation sets of $\kp_s$. We find that $I(2)$ consists of the germ of the
curve $(\frac{70}{3}c_5t^3+O(t^4), -10c_5t^2+O(t^3))$ with $t\in
\R,0$ and $V(2)$ is given by the germ of the curve
$\displaystyle(s_1,\frac{2c_{410}^2}{5c_5}s_1^2 + O(s_1^3))$ with
$s_1\in\R,0$. Thus, the stratification of the parameter space
$(s_1,s_2)$ of $\g_s$ is as shown in Figure \ref{fig:inf.ord3}
center.

As $\g_s$ is \textit{FR}-generic, by Remark \ref{teo:nlightlike}, the
families $\kp_s$ and $\kp_s^\prime$ are $\mathcal{R}$-versal
deformations of an $A_2$ and an $A_1$ singularity, respectively. The bifurcations in the graph of the curvature function $\kp_s$ are
as shown in Figure \ref{fig:bif.graf-curvatura-I3}, where we
indicate the inflections ($\kp_s = 0$) and vertices ($\kp_s^\prime =
0$) of $\g_s$.

\begin{figure}[h!]
\centering
\includegraphics[scale = 0.5]{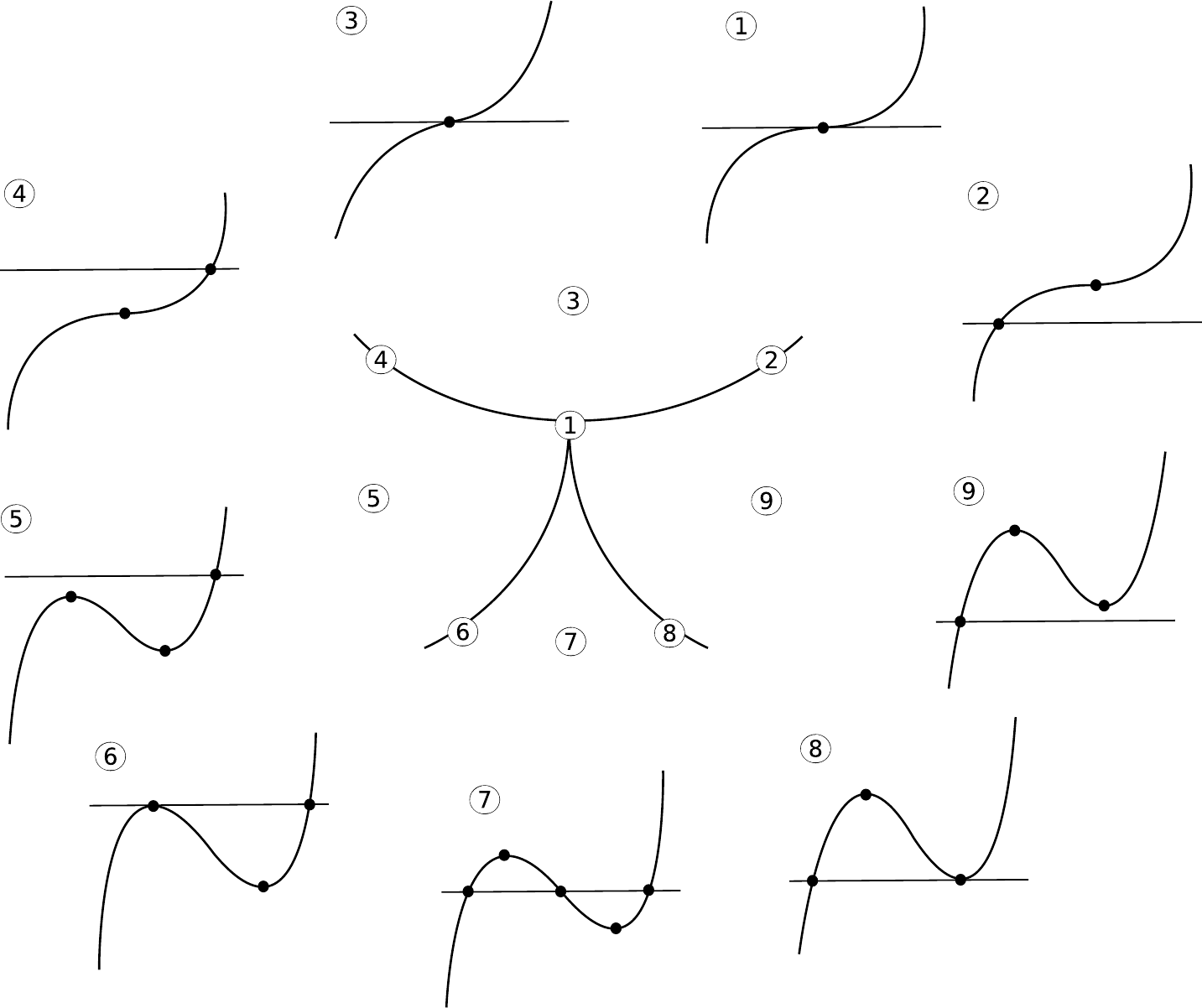}
\caption{\small{Bifurcations in the graph of $\kp_s$. The horizontal
line is the zero level.}}\label{fig:bif.graf-curvatura-I3}
\end{figure}

Using the behavior of the evolute at an inflection of order $3$, the
number and relative position of vertices and inflections of $\g_s$
obtained above and the previous results we conclude that the
bifurcations in $\g_s$ and its evolute are as shown in Figure
\ref{fig:inf.ord3}.

(iii) The computations in (ii) depend only on the fact that the
curve has an inflection of order $3$ and on the family being an
\textit{FR}-generic deformation. The family $\a_u$ satisfies these conditions
and can be taken as an \textit{FR}-model.
\end{proof}

\begin{figure}[h!]
\centering
\includegraphics[scale = 0.7]{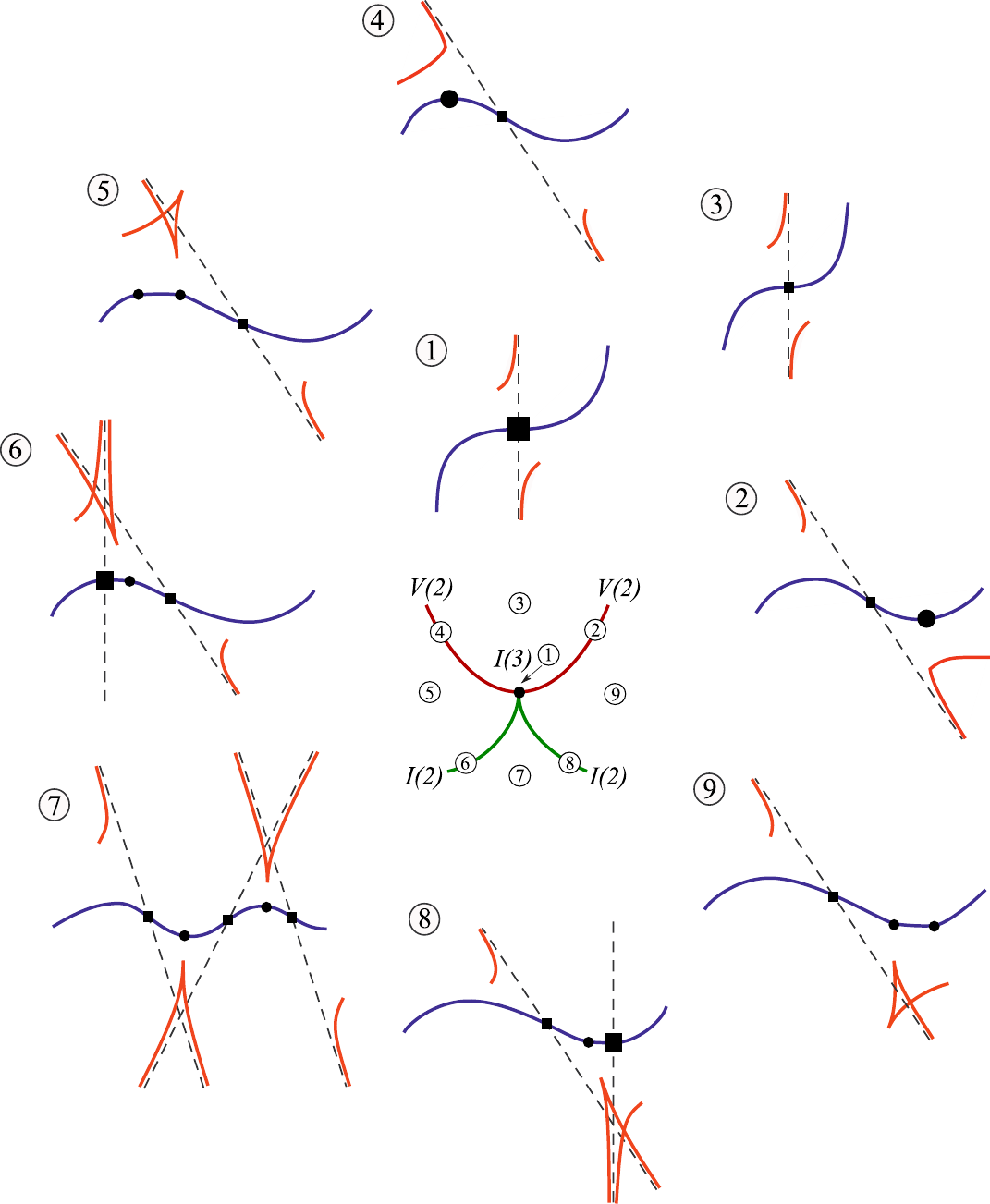}
\caption{\small{\textit{FR}-generic deformations of a curve and its evolute at an inflection of order 3.
The central figure is the stratification of the parameter space.}}\label{fig:inf.ord3}
\end{figure}

%%%%%%%%%%%%%%%%%%%%%%%%%%%%%%%%%%%%%%%%%%%%%%%%%%%%%%%%%%%%%%%%%%%%%%%%%%%%%%%%%%%%%%%%%%%%%%%%%%%%%%%%%%%%%%%%%%%%%%%%
%%%%%%%%%%%%%%%%%%%%%%%%%%%%%%%%%%%%%%%%%%%%%%%%%%%%%%%%%%%%%%%%%%%%%%%%%%%%%%%%%%%%%%%%%%%%%%%%%%%%%%%%%%%%%%%%%%%%%%%%
\section{\bf Geometric deformations of a regular curve at a lightlike point}

In this section, we deal with lightlike inflections of finite order of regular curves, so an \textit{FRLS}-generic family (see Definition \ref{def:FRLS-def.}) will be called \textit{FRL}-generic. We suppose that the tangent direction is along $(1,1)$ at the origin and take $\g_s$ in the form 
$\g_s(t) = (t, f_s(t))$, with $\g_0 =\g$, $f_0 = f$, $f(0) = 0$ and $f^\prime(0) = 1$.

When deforming the curve, the lightlike points on $\g_s$ are points
of tangency with a line parallel to $y = x$. To capture this
contact, we use the subgroup $\mathcal{K}^*$ of the contact group
$\mathcal{K}$ formed by changes of coordinates in the source that
preserve parallel lines. In order to simplify the calculations we
apply the transformation $T(x,y) = (x,x-y)$ and consider
$\mathcal{K}^*$ preserving horizontal lines. (For more details about
the group $\mathcal{K}^*$ and the classification of function germs
$\R^2,0\rightarrow\R,0$ under its action, see
\cite{Damon,Regilene}.)

Applying $T$ on $\g_s$ gives $\overline{\g_s}(t) = (t,t-f_s(t))$. We can consider,
locally, $\overline{\g_s} =
G_s^{-1}(0)$, where $G:\R_1^2\t\R,(0,0)\rightarrow\R$ is given by
$G(x,y,s) = y-x+f_s(x)$.

The \textit{extended tangent space} $L\mathcal{K}_e^*\cdot g$, with
$g\in\M_2$, is defined as $L\mathcal{K}_e^*\cdot g =
\E_2\cdot\{\frac{\partial g}{\partial x}\} +
\E_y\cdot\{\frac{\partial g}{\partial y}\} + g^*(\M_1)\cdot\E_2$,
where $\E_y$ denotes the set of germs in $\E_2$ which depend only on
$y$. As $\mathcal{K}^*$ is a geometric subgroup of $\mathcal{K}$
(\cite{Damon}), we have the following result.

\begin{theorem}\label{teo:def:K-versal}{\rm(\cite{Damon})}
An $m$-parameter deformation $G$ of a function germ $g\in
\M_2$ is a $\mathcal{K}_e^*$-versal deformation if, and
only if, $L\mathcal{K}^*_e\cdot g +
\R\cdot\{\dot{G_1},\dots,\dot{G_m}\} = \E_2$, where $\dot{G_i}$ is the derivative of $G$ in relation to the $i$-th parameter.
\end{theorem}

Consider $\g$ a regular curve with a lightlike inflection of order
$k-1$ at $t_0 = 0$, with $k\geq 2$, and $\g_s$ an $m$-parameter
family of regular curves, with $\g_0 = \g$.

\begin{theorem}\label{teo:gener.trans}
Let $\g_0$ be a regular curve with a lightlike inflection of order
$k-1$ at $t_0$, with $k\geq 2$, and $\g_s$ be an $m$-parameter
family of regular curves containing $\g_0$, with $m\> k-1$. Then, the family $G$ associated to $\g_s$ is a $\mathcal{K}_e^*$-versal deformation of $G_0$ at $t_0$
if, and only if, $j^n\Phi$ is transverse to the stratum of lightlike
inflections of order $k-1$ in $J^n(1,1)$, with $n\geq k$.
\end{theorem}
\begin{proof} We start by proving the theorem for $\g_s$ an $(k-1)$-parameter family.

The stratum $LI(k-1)$ is given
by $a_1 = 1, a_2 = a_3 = \cdots = a_k = 0$ and its tangent space by
$\R\cdot\{e_{k+1},...,e_n\}$. We have $j^n\Phi(0,0) =
(1,0,...,0,c_{k+1},c_{k+2},...,c_n)$, with $c_{k+1}\neq 0$, where
$\g_0(t) = (t,t + c_{k+1}t^{k+1} + c_{k+2}t^{k+2} + O(t^{k+3}))$.
Hence, $j^n\Phi$ is transverse to $LI(k-1)$ if, and only if,
$Im(d_{(0,0)}j^n\Phi) + \R\cdot\{e_{k+1},...,e_n\} = \R^n$.

%We denote $Im(d_{(0,0)}j^n\Phi)$ and $\R\cdot\{e_{k+1},...,e_n\}$,
%respectively, by $U$ and $W$.

%Therefore, its enough to show that there are vectors
%$u^1,...,u^k\in\R^k$ such that $d_{(0,0)}j^n\Phi(u^l) =
%(0,\dots,0,1,0,\dots,0,\underline{\ \ \ \ })$, with $1$ in the
%$l$-th enter and $l = 1,\dots,k$.

Note that the jacobian matrix of $j^n\Phi$ at $(0,0)$ is given by
$$Jac(j^n\Phi)(0,0) = \left(
    \begin{array}{cc}
       0 & \frac{1}{l!}\frac{\partial^{l+1} f}{\partial s_i\partial t^l}(0,0) \\
      (j+1)c_{j+1} & \frac{1}{j!}\frac{\partial^{j+1} f}{\partial s_i\partial t^j}(0,0) \\
    \end{array}
  \right)_{n\t k},
$$
with $l = 1,\dots,k-1$ and $j = k,\dots, n$. We can show that
$e_1,..., e_k\in Im(d_{(0,0)}j^n\Phi) + \R\cdot\{e_{k+1},...,e_n\}$
if, and only if, the matrix $\left(
\begin{array}{c}
\frac{1}{j!}\frac{\partial^{j+1} f}{\partial s_i\partial t^j}(0,0) \\
\end{array}\right)$ has rank $k-1$, which is equivalent to $\g_s$ be a $\mathcal{K}_e^*$-versal deformation of $\g_0$.

%Thus, each one of the equations described previously becames a
%system of equations as we describe bellow for the $l$-th equation:
%$$\left\{\begin{array}{c}
%    \displaystyle\frac{1}{j!}\sum_{i=1}^{k-1} u^l_{i+1}\frac{\partial^{j+1} f}{\partial s_i\partial t^j}(0,0)
%    = \chi_{j,l}\\
%    (k+1)c_{k+1}u^l_1 + \displaystyle\frac{1}{k!}\sum_{i=1}^{k-1} u^l_{i+1}\frac{\partial^{k+1} f}{\partial s_i\partial t^k}(0,0) = 0
%  \end{array}\right.,$$
%with $j = 1,\dots,k-1$ and $\chi_{j,l} = 0$, if $j\neq l$ and
%$\chi_{j,l} = 1$, if $j = l$.
%
%Hence, to show that $e_1,...,e_k$ belongs to $U+W$, it is enough to
%show that the obtained systems have solution. If we consider the
%first $k-1$ equations of the first $k-1$ systems, we observe that
%they have solution if, and only if, the matrix $A = \left(
%    \begin{array}{c}
%      \frac{1}{j!}\frac{\partial^{j+1} f}{\partial s_i\partial t^j}(0,0) \\
%\end{array}\right)$, with $j = 1,\dots,k-1$, has rank $k-1$. This part of
%the system provides $u_2^l,...,u_k^l$, with $l=1,...,k-1$. Using
%that $c_{k+1}\neq 0$ and the least equation of the $j$-th system, we
%obtain $u_1^l$ in function of $u_2^l,...,u_k^l$, finding in this way
%$u^l$, with $l=1,...,k-1$. Moreover, the $k$-th system is easily
%solved by $u^k = (\frac{1}{(k+1)c_{k+1}},0,...,0)$.
%
%Therefore, $j^n\Phi$ is transverse to the stratum $LI(k-1)$ if, and
%only if, the matrix $A$ has rank $k-1$, which is the condition to
%$\g_s$ be a $\mathcal{K}_e^*$-versal deformation of $\g_0$.
%
The proof when $\g_s$ is an $m$-parameter family, with $m\> k-1$, follows
in a similar way.
\end{proof}

Note that the stratum $LI(k-1)$ is contained in the strata $I, I(2),
\dots, I(k-1)$, $L, LI, LI(2),\dots, LI(k-2)$, $V, V(2), \dots,V(2(k-1))$.
Therefore, if $ j^n\Phi$ is transverse to the stratum $IL(k-1)$ it will also be transverse to these strata.

\begin{corollary}
A deformation
$\g_s$ of a regular curve $\g$ at a lightlike inflection of finite
order is \textit{FRL-generic} if, and only if, the family $G$ associated to $\g_s$ is a
$\mathcal{K}_e^*$-versal deformation of $G_0$ at $t_0$.
\end{corollary}

\begin{theorem}\label{teo:IL} Let $\g$ be a regular curve with a lightlike ordinary inflection at $t_0$.
\begin{itemize}
\item[\text{(i)}] The caustic of $\g$ is the union of the tangent line
of $\g$ at $t_0$ and a smooth curve with a
lightlike ordinary inflection at $\g(t_0)$ which lies in the two
quadrants determined by the lightlike lines which do not contain the
curve $\g$, as shown in Figure \ref{fig:inf.ord.light} center.

\item[\text{ (ii)}] For an \textit{FRL}-generic $1$-parameter family $\g_s$
with $\g_0=\g$, the bifurcations in the caustic are as shown in
Figure \ref{fig:inf.ord.light}. We have two lightlike points and an
ordinary inflection on one side of the transition and two outward
vertices and an ordinary inflection on the other side of the
transition.

\item[\text{(iii)}] Any \textit{FRL}-generic $1$-parameter family of $\g$ is \textit{FRL}-equivalent to the model family $\a_u(t) = (t, (1+u)t + t^3)$.
\end{itemize}
\end{theorem}

\begin{figure}[h!]
\centering
\includegraphics[scale = 0.8]{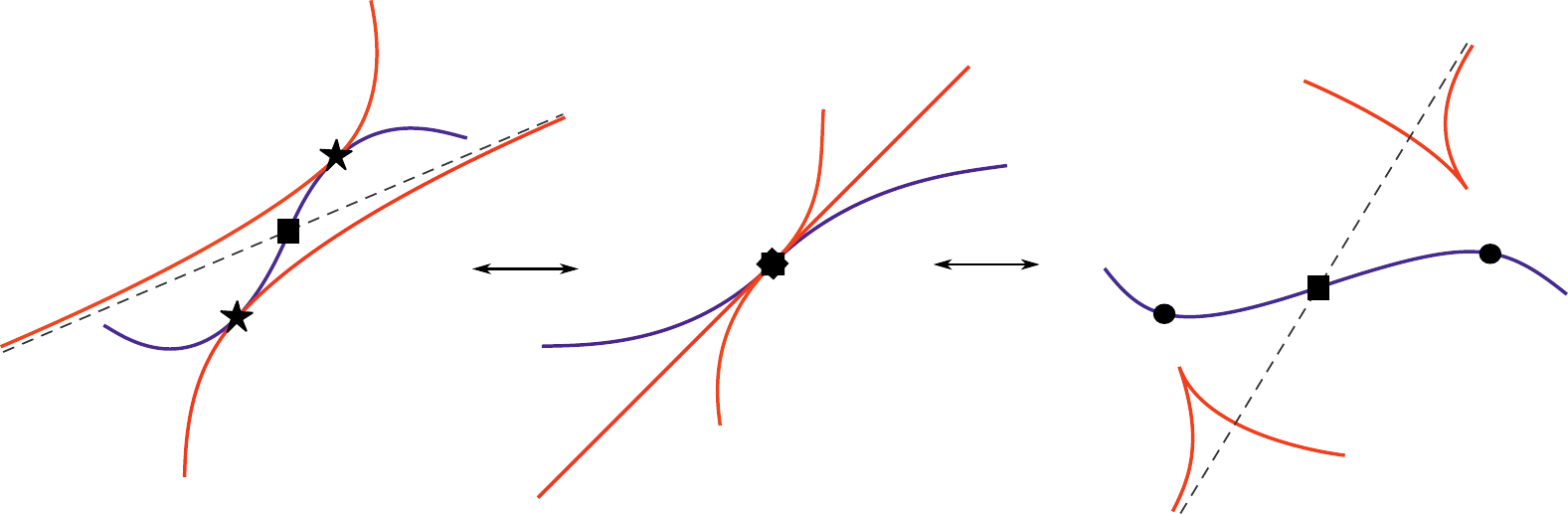}
\caption{\small{\textit{FRL}-generic deformations of a curve and its caustic at a lightlike ordinary
inflection.}}\label{fig:inf.ord.light}
\end{figure}

\begin{proof}
(i) This is a consequence of Theorem 3.6 in
\cite{reeve}.

%The equation in the definition of the caustic is $ \lambda
%f^{\prime\prime}(t) - 1 + f^\prime(t)^2 = 0$. If $t = t_0 = 0$,
%follows that $f^{\prime\prime}(0) = 0$ and $f^\prime(0) = 1$ and,
%consequently, the above equation is satisfied for any
%$\lambda\in\R$, that is, $\{\g(0) - \lambda\g^\prime(0)^\perp\} =
%\{(0,0) - \lambda(1,1)\} = \{\lambda(-1,-1)\}$ is part of the
%caustic of $\g$.

%Now, if $t\neq 0$, then, locally, $f^{\prime\prime}(t)\neq 0$. Thus,
%the equation of the caustic becomes $\lambda =
%\frac{1-(f^\prime(t))^2}{f^{\prime\prime}(t)}$.

%Therefore the caustic of $\g$ is given by the tangent line of $\g$
%at $t_0$ (lightlike line) and the curve $(t(2 - \frac{2c_4}{3c_3}t +
%(9c_3 - \frac{10c_5}{3c_3} + \frac{8c_4^2}{3c_3^2})t^2 + O(t^3)),
%t(2 - \frac{2c_4}{3c_3}t + (4c_3 - \frac{10c_5}{3c_3} +
%\frac{8c_4^2}{3c_3^2})t^2 + O(t^3)))\}$ that can be reparametrized
%by $(t, t -\frac{5}{8}c_3t^3 + O(t^4))$, as shown in Figure
%\ref{fig:inf.ord.light} center.

(ii) As $\g_s$ is an
\textit{FRL}-generic deformation of $\g$, we write it in the form $\g_s(t) =
(t,f_s(t))$ with $f_s(t) = (1+s)t + \overline{c_2}(s)t^2 + (c_3 +
\overline{c_3})t^3 + O_s(t^4)$, where $\overline{c_i}(0) = 0$ for $i
= 2,3$ and $c_3\neq 0$.

The family of distance squared function $d$ is an
$\mathcal{R}^+$-versal deformation of an $A_3$ singularity. That
implies that the bifurcation set $Bif(d)$ of $d$ is locally
diffeomorphic to a cuspidal edge. Hence the individual caustics of
$\g_s$ are obtained by taking sections of the cuspidal edge (see
\cite{Arnold}). One can show that the plane $s = 0$ has a Morse
contact with the singular set of $Bif(d)$ and intersects $Bif(d)$ along two tangential curves. Thus the caustics
undergoes the beaks transitions. It follows that $\g_s$ has two vertices on one side of the transition and
none on the other. 

To obtain the inflections and lightlike points on
$\g_s$, we consider a $2$-dimensional section of the jet space given
by $(a_3,\dots,a_k) = (c_3,\dots,c_k)$. In such section, $V$ is a
regular curve given by $(1 + \frac{2}{c_3}a_2^2 + O(a_2^3), a_2)$.
Moreover, for $s = 0$, we can write the image of the restriction of
the Monge-Taylor map to this section in the form $(1 +
\frac{1}{3c_3}t^2 + O(t^3),t)$. Comparing such curves we conclude
that this section is as described in Figure \ref{fig:sec2-inf-light}
center.

\begin{figure}[h!]
\centering
\includegraphics[scale = 0.8]{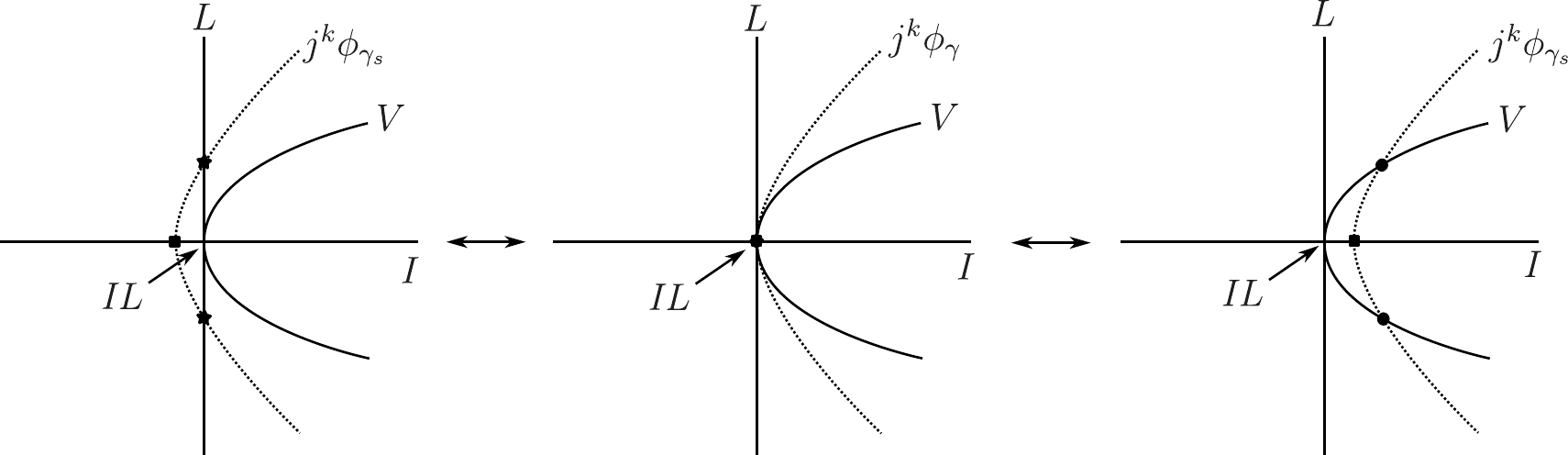}
\caption{\small{A $2$-dimensional transversal section of the
stratification of $J^k(1,1)$ by $I, L, LI, V$ and the relative
positions of the image of the Monge-Taylor map of $\g_s$ for $s < 0$
in left, $s = 0$ in center and $s > 0$ in right (case $c_3 >
0$).}}\label{fig:sec2-inf-light}
\end{figure}

We conclude that on one side of the transition we have an inflection
between two lightlike points and on the other side we have an
inflection between two vertices (see Figure
\ref{fig:sec2-inf-light}). Moreover, when $\g_s$ has vertices, the curvature function has a maximum and a minimum between
a zero. Consequently, $\kp_s\kp^{\prime\prime}_s < 0$ at both critical
points of $\kp_s$, that is, both vertices of $\g_s$ are outward.

(iii) The computations in (ii) depend only on the fact that the
curve has a lightlike ordinary inflection and on the family being an
\textit{FRL}-generic deformation. The family $\a_u$ satisfies these
conditions and can be taken as an \textit{FRL}-model.
\end{proof}

\begin{theorem}\label{teo:IL(2)} Let $\g$ be a regular curve with a lightlike inflection of order $2$ at $t_0$.
\begin{itemize}
\item[\text{(i)}] The caustic of $\g$ is the union of the tangent line of $\g$ at $t_0$ and a smooth curve with a lightlike
inflection of order $2$ which lies in the
semi-plane determined by the lightlike line that contains the curve
$\g$, see Figure \ref{fig:bif.IL(2)} $\textcircled{1}$.

\item[\text{ (ii)}] For an \textit{FRL}-generic $2$-parameter family $\g_s$ with
$\g_0=\g$, the bifurcations in the caustic are as shown in Figure
\ref{fig:bif.IL(2)}.

\item[\text{ (iii)}] Any \textit{FRL}-generic $2$-parameter family of $\g$ is \textit{FRL}-equivalent to the model family $\a_u(t) = (t, (1+u_1)t + u_2t^2 + t^4)$.
\end{itemize}
\end{theorem}

\begin{figure}[h!]
\centering
\includegraphics[scale = 0.65]{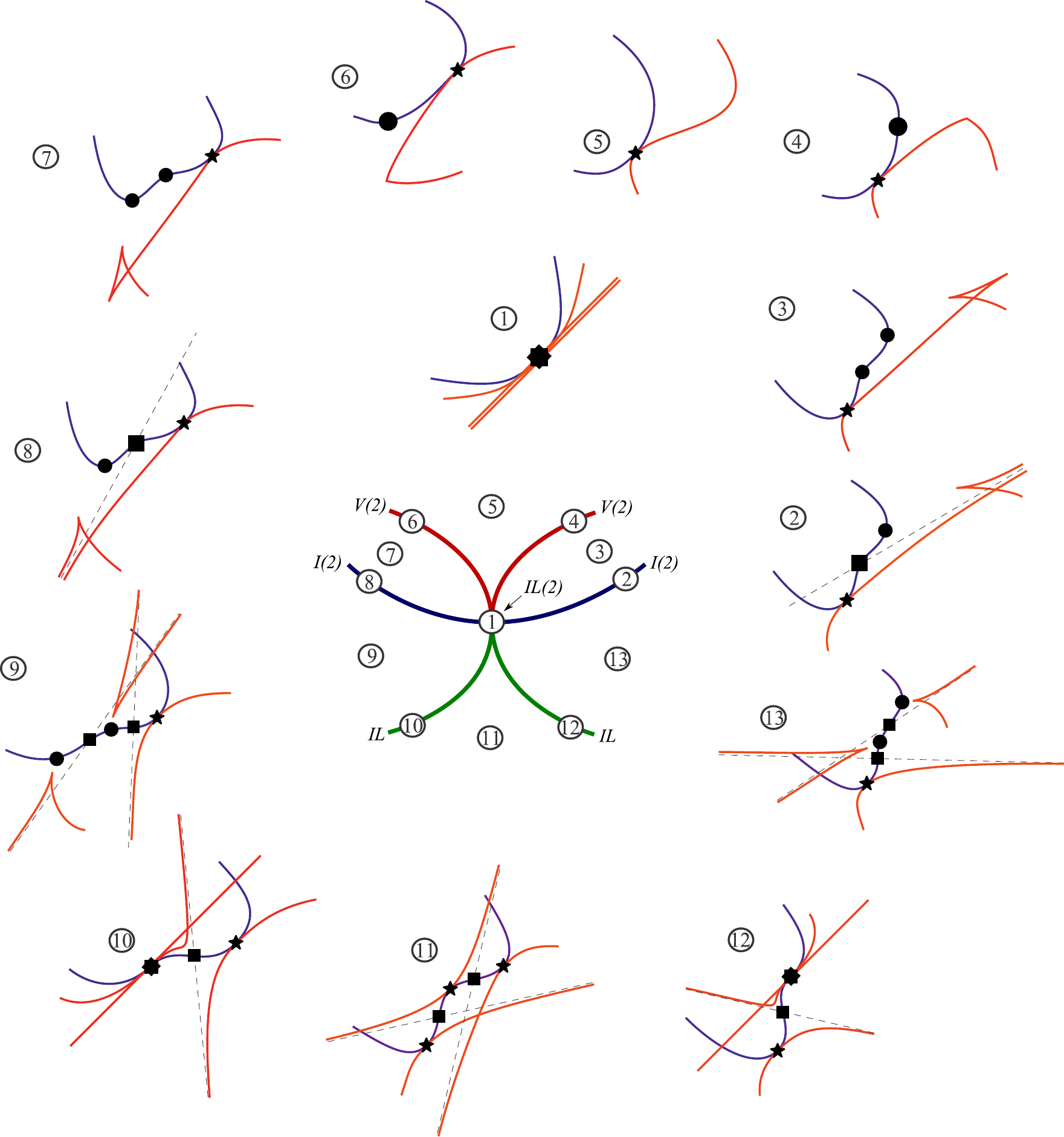}
\caption{\small{\textit{FRL}-generic deformations of a curve and its caustic at a lightlike inflection
of order $2$. The central figure is the bifurcation set in the
parameters space.}}\label{fig:bif.IL(2)}
\end{figure}

\begin{proof}
As $t_0$ is a lightlike inflection of order $2$, we take $\g(t) =
(t,t + c_4t^4 + c_5t^5 + c_6t^6 + c_7t^7 + O(t^8))$, with $c_4\neq
0$, at $t_0 = 0$.

 (i) The result is a consequence of the Theorem 3.6 in
\cite{reeve}.

(ii) Firstly, we analyze the deformations with respect to the
lightlike points and inflections. As $\g_s$ is an \textit{FRL}-generic
family, we can write $\g_s(t) = (t,f_s(t))$, with $f_s(t) = (1+
s_1)t + s_2t^2 + \overline{c_3}(s)t^3 + (c_4+\overline{c_4}(s))t^4 +
O_s(t^5)$ and $\overline{c_i}(0) = 0$ for $i = 3,4$. The zeros of
the function $g_s(t) = 1-f_s^\prime(t)$ give the lightlike points of
$\g_s$ and the singular points of $g_s$ give the inflections of
$\g_s$. We show that $g_s$ is an $\mathcal{R}^+$ and
$\mathcal{R}$-versal deformation of an $A_2$ singularity. Thus, we
have models for the zero, discriminant, singular and bifurcation
sets of $g_s$. We have
\begin{eqnarray}
\mathcal{D}_{g_s}: & & \{(8c_4t^3+O(t^4),-6c_4t^2+O(t^3))\in\R^2;\ t\in\R,0\}\\
Bif(g_s): & & \displaystyle\{(s_1,\frac{3c_{310}^2}{4c_4}s_1^2 +
O(s_1^3))\in\R^2;\ s_1\in\R,0\}.
\end{eqnarray}
where $\mathcal{D}_{g_s}$ denotes the discriminant set of $g_s$. Note that $\mathcal{D}_{g_s}$ and $Bif(g_s)$ correspond,
respectively, to the strata $LI$ and $I(2)$ of $\g_s$.

We obtain deformations of the lightlike points and inflections using
the deformations in the graph of $g_s$ (Figure
\ref{fig:bif.graf-curvatura-I3}).

For the vertices we use the family $G(t,s_1,s_2) =
f^{\prime\prime\prime}(t,s_1,s_2)(1-(f^\prime(t,s_1,s_2))^2) +
3f^\prime(t,s_1,s_2)(f^{\prime\prime}(t,s_1,s_2))^2$, whose
discriminant set is the stratum $V(2)$. The
function $G_0$ has an $A_3$ singularity at $t_0 = 0$. However, $G$
is not an $\Rl$-versal deformation of $G_0$ at $t_0$, it is
induced by the $\Rl$-versal deformation $F(x,u,v,w) = x^4 + ux^2 +
vx + w$, that is, there are germs of maps $a:\R\t\R^2,(0,0)\rt\R$
and $b:\R^2,0\rt\R^3,0$ such that $a(t,0) = \ph(t)$ is a
diffeomorphism and
$$G(t,s) =
F(a(t,s),b(s)) = (a(t,s))^4 + b_1(s)(a(t,s))^2 + b_2(s)a(t,s) +
b_3(s),$$ where $b = (b_1,b_2,b_3)$. Note that if $s$ belongs to the
discriminant set of $G$, then $b(s)$ belongs to the discriminant set
of $F$ which is a swallowtail.

We can set, $b_1(s_1,s_2) = s_2, \ b_2(s_1,s_2) = s_1$ and we get
$$b_3(s_1,s_2) = \b_1s_1^2 + \b_2s_1s_2 + \frac{5}{4}s_2^2 +
\b_3s_1^3 + \b_4s_1^2s_2 + \b_5s_1s_2^2 - (\frac{88c_4c_6 -
149c_5^2}{16c_4^2})s_2^3 + O_4(s_1,s_2).$$ Therefore, the image of
$b$ is the a graph of the function $b_3(s_2,s_1)$, and is also the
zero set of the function $h(x,y,z) = z - b_3(x,y)$. We can use the
classification of submersions from $\R^3$ to $\R$ up to diffeomorphisms in the source which preserve the swallowtail $X =
\mathcal{D}_F$ (the $\Rl(X)$-classification) obtained in 
\cite{Arnold, Transversalidade, Meu-artigo} and show that $h$ is
$\Rl(X)$-equivalent to
$$\overline{h}(x,y,z) = z - \frac{5}{4}x^2 + (\frac{88c_4c_6 -
149c_5^2}{16c_4^2})x^3.$$ Thus the intersection of the swallowtail
and the fiber $h^{-1}(0)$ is diffeomorphic to \linebreak $\{(2t^2,-8t^3,5t^4);
t\in \R\}\cup\{(-\frac{6}{5}t^2,-\frac{8}{5}t^3,\frac{9}{5}t^4);
t\in \R\}$ which consists of two tangential curves at the origin.

From Theorem \ref{teo:IL-V}, the
parametrization of $LI$ is part of the parametrization of $V(2)$.

Looking the intersection of the fiber of $\overline{h}$ and the
swallowtail we obtain the number of zeros of $G_{(s_1,s_2)}$, that
is, the number of vertices in each stratum. We conclude
that the \textit{FRL}-deformations of $\g$ are as shown in Figure
\ref{fig:bif.IL(2)}.

(iii) The statement follows from the fact that $\a_u$ is an
\textit{FRL}-generic deformation at a lightlike inflection of order $2$ and
the calculations in (ii) depend only on the fact that the curve has
a lightlike inflection of order $2$ and on the deformation being
\textit{FRL}-generic.
\end{proof}

%%%%%%%%%%%%%%%%%%%%%%%%%%%%%%%%%%%%%%%%%%%%%%%%%%%%%%%%%%%%%%%%%%%%%%%%%%%%%%%%%%%%%%%%%%%%%%%%%%%%%%%%%%%%%%%%%%%%%%%%
%%%%%%%%%%%%%%%%%%%%%%%%%%%%%%%%%%%%%%%%%%%%%%%%%%%%%%%%%%%%%%%%%%%%%%%%%%%%%%%%%%%%%%%%%%%%%%%%%%%%%%%%%%%%%%%%%%%%%%%%
\section{\bf Geometric deformations of curves with a non-lightlike ordinary cusp}
%%%%%%%%%%%%%%%%%%%%%%%%%%%%%%%%%%%%%%%%%%%%%%%%%%%%%%%%%%%%%%%%%%%%%%%%%%%%%%%%%%%%%%%%%%%%%%%%%%%%%%%%%%%%%%%%%%%%%%%%
%%%%%%%%%%%%%%%%%%%%%%%%%%%%%%%%%%%%%%%%%%%%%%%%%%%%%%%%%%%%%%%%%%%%%%%%%%%%%%%%%%%%%%%%%%%%%%%%%%%%%%%%%%%%%%%%%%%%%%%%

Suppose that $\g(t) =
(\a(t),\b(t))$ has a cusp singularity at the origin and that the limiting tangent direction to $g$ at that point is timelike (the spacelike case follows analogously). We call this singularity \textit{a non-lightlike cusp}. Here, the Monge-Taylor map does not intersect the stratum involving lightlike point, so an \textit{FRLS}-generic family will be called \textit{FRS}-generic. Denote by 
$(a_1,\dots,a_k;b_1,\dots,b_k)$ the coordinates in $J^k(1,2)$. The
strata of interest are:
\begin{eqnarray}
\mbox{Cusp \ ($C$)}: & & a_1 = 0 \ \text{and} \ b_1 = 0
\nonumber\\
\text{Lightlike \ ($L$)}: & & a_1 \pm b_1 = 0 \nonumber\\
\text{Inflections \ ($I$)}: & & a_1b_2 - a_2b_1 = 0  \nonumber\\
\text{Vertices \ ($V$)}: & & (a_1^2-b_1^2)(a_1b_3-a_3b_1) +
2(b_1b_2-a_1a_2)(a_1b_2-a_2b_1) = 0\nonumber
\end{eqnarray}

Suppose $\g$ has
an ordinary cusp singularity at $t_0=0$, that is, $\g$ is
$\mathcal{A}$-equivalent to $(t^2,t^3)$. We can take, without loss of generality, $\g^{\prime\prime}(0)$ parallel to $(1,0)$ and write $\g(t)
= (t^2,\b(t)) = (t^2, c_3t^3 + O(t^4))$, with $c_3\neq 0$. The
strata $I, V$ and $L$ are submanifolds of $J^k(1,2)$ of codimension
1 which contain the submanifold $C$ of codimension 2. The stratum
$I$ is smooth along $C$, the stratum $L$ is the union of two
transverse components $L_1$ and $L_2$, represented, respectively, by
the equations $a_1 + b_1 = 0$ and $a_1 -b_1 = 0$, and the stratum
$V$ is the union of two smooth submanifolds $V_1$ and $V_2$ which
intersects transversally along $C$.

We can write any $1$-parameter family of curves $\g_s$ with $\g = \g_0$
in the form
\begin{equation}\label{eq:01}
\g_s(t) = (t^2, \b(t,s)) = (t^2,\overline{c_1}(s)t +
\overline{c_2}(s)t^2 + \overline{c_3}(s)t^3 + O_s(t^4)),
\end{equation}
with $\overline{c_1}(0) = \overline{c_2}(0) = 0$, $\overline{c_3}(0)
= c_3\neq 0$, $\overline{c_k}(0) = c_k$ for $k\> 4$ and $s\in\R,0$.

Our goal is to find conditions to the \textit{FRS}-genericity of $\g_s$. Note
that if the image of $d(j^k\Phi)(0,0)$ is transverse to the stratum
$C$, then it will be transverse to $I, L$ and $V$. The family of
Monge-Taylor maps $j^k\Phi$ is transverse to the stratum $C$ if, and
only if, $\displaystyle\frac{\partial^2 \b}{\partial s\partial
t}(0,0) \neq 0$, which is exactly the condition for $\g_s$ to be an
$\mathcal{A}_e$-versal deformation of the ordinary cusp singularity
of $\g$.

\begin{corollary}
Let $\g$ be a germ of a non-lightlike ordinary cusp and $\g_s$ be an
$1$-parameter family of curves with $\g_0 = \g$. Then $\g_s$
is an \textit{FRS-generic} deformation of $\g$ if, and only if, $\g_s$ is an
$\mathcal{A}_e$-versal deformation of $\g$.
\end{corollary}

\begin{theorem}\label{teo:bif-cusp}
Let $\g$ be a germ of a non-lightlike ordinary cusp at $t_0$.
\begin{itemize}
\item[\text{(i)}] The caustic of $\g$ is the union of the limiting normal line of $\g$ at $t_0$
and a smooth curve which has ordinary tangency at $\g(t_0)$ with the
limiting normal line of $\g$ at $t_0$. Moreover, that curve lies in
the semi-plane determined by the normal line of $\g$ at $t_0$ which
does not contain the curve $\g$, see Figure
\ref{fig:def-cusp} center.

\item[\text{ (ii)}] For an \textit{FRS}-generic $1$-parameter family $\g_s$
with $\g_0=\g$, the bifurcations in the curve and its caustic are
as shown in Figure \ref{fig:def-cusp}.

\item[\text{(iii)}] Any \textit{FRS}-generic $1$-parameter family of $\g$ is \textit{FRS}-equivalent to the model family $\a_u(t) = (t^2, ut + t^3)$.

\end{itemize}
\end{theorem}

\begin{figure}[h!]
\centering
\includegraphics[scale = 0.87]{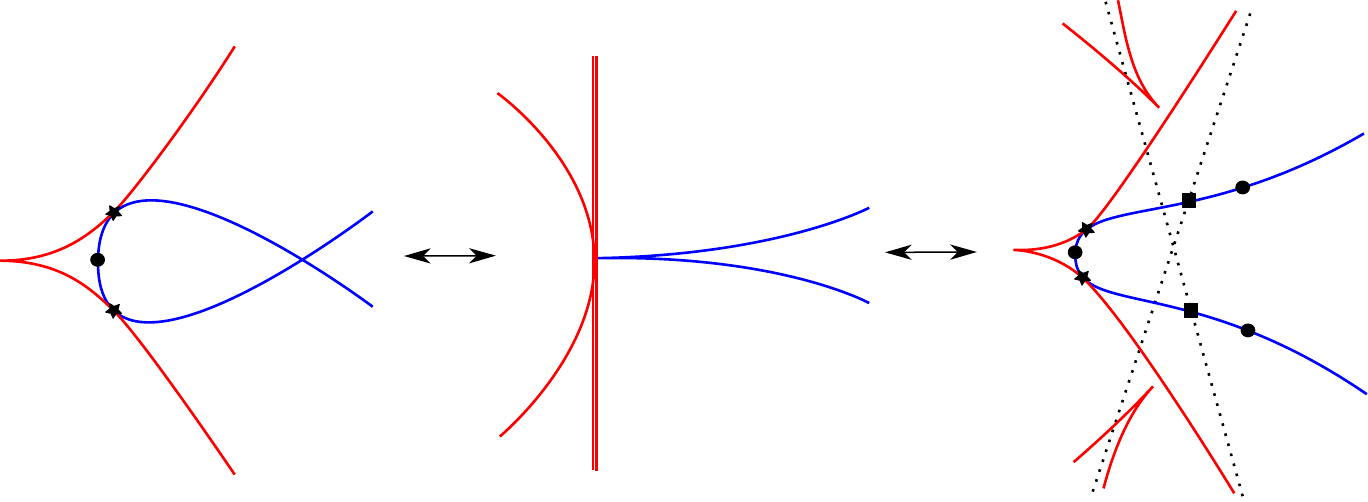}
\caption{\small{\textit{FRS}-generic deformations of a curve and its caustic at a non-lightlike ordinary
cusp.}}\label{fig:def-cusp}
\end{figure}

\begin{proof}
\text{(i)} We have
$$d_{u}(t) = -(t^2-u_1)^2 + (\b(t)-u_2)^2 = -t^4 + 2u_1t^2
+ -u_1^2 + \b^2(t) - 2u_2\b(t) + u_2^2.$$ Isolating $u_1$ in
$d_{u}^{\prime\prime} = 0$ and substituting in $d_u^\prime = 0$ we
get
$$t^2(4t + u_2\frac{(-\b^\prime(t) + \b^{\prime\prime}(t)t)}{t^2} + \frac{\b(t)(\b^\prime(t) - \b^{\prime\prime}(t)t) - {\b^\prime}^2(t)t}{t^2}) = 0.$$
As $t = 0$ is a root of the above equation, the line $u_1 = 0$ is
part of the caustic of $\g$.

For $t\neq 0$, we have $u_2 = \displaystyle \frac{4t^3 +
(-\b(t)\b^{\prime\prime}(t) - {\b^\prime}^2(t))t +
\b(t)\b^{\prime\prime}(t)}{\b^\prime(t) - \b^{\prime\prime}(t)t}$.

Therefore, the caustic of $\g$ is the union of the line $u_1 = 0$ and
the curve parametrized by $\displaystyle t\mapsto (-t^2 + O(t^3),
\frac{-4}{3c_3}t + O(t^2))$.

\text{ (ii)} We take $\g_s(t) = (t^2,\b(t,s)) = (t^2, st +
\overline{c_2}(s)t^2 + \overline{c_3}(s)t^3 + O_s(t^4))$.

%\np \textit{Lightlike points:} They are given by
%$\displaystyle\b_s^\prime(t) = \frac{\partial \b}{\partial t}(t,s) =
%\pm 2t$. Consider the families $h^{\mp}(t,s) = \b_s^\prime(t)\mp 2t
%= s + 2(\overline{c_2}(s)\mp 1)t + 3\overline{c_3}(s)t^2 +
%O_s(t^3)$, which are deformations of $h_0^{\mp}(t) = \mp 2t +
%3c_3t^2 + O(t^3)$. Note that $h_0^{\mp}$ is regular at $t=0$. Thus,
%for each $s\neq 0$, the deformation $h_s^-$ has one zero and $h_s^+$
%has another zero. The zeros of $h_s^{\mp}$ and, consequently, the
%lightlike points of $\g_s$, are given by $t = \displaystyle \pm
%\frac{1}{2}s + O(s^2)$, that is, for $s\neq 0$, there are two
%lightlike points on $\g_s$.

%\np \textit{Inflexões:} Analyzing the expression of $\kp_s$ we
%observe that its denominator vanish only on lightlike and singular
%points of $\g$. However, for $s\neq 0$ the curve $\g_s$ is regular
%close to $t_0 = 0$ and the lightlike points obtained above do not
%vanish the numerator. Hence, to find the inflections of $\g_s$, it
%is enough to find the zeros of the numerator of the curvature
%function $\kp_s$. For this consider the família $G(t,s) =
%-\b_s^\prime(t) + t\b_s^{\prime\prime}(t)$ which is half of the
%numerator of $\kp_s(t)$. Thus, $G_s(t) = -s + 3\overline{c_3}(s)t^2 +
%O_s(t^3)$. Note that $\displaystyle\frac{\partial G}{\partial
%s}(0,0) = -1\neq 0$, that is, $G_s$ is an $\Rl$-versal deformation
%of $G_0(t) = 3c_3t^2 + O(t^3)$, which has an $A_1$ singularity.
%Therefore, if $c_3 > 0$, $G_s$ has $2$ zeros for $s>0$ and none for
%$s<0$, and if $c_3<0$, $G_s$ has $2$ zeros for $s<0$ and none for
%$s>0$.

We first take the $2$-dimensional section of the stratification of
$J^k(1,2)$ by $C,L,I,V$ obtained by fixing
$(a_2,a_3,\dots,a_k;b_2,b_3,\dots,b_k) =
(1,0,\dots,0;0,c_3,\dots,c_k)$, with $c_3\neq 0$. In this section
the stratum $I$ is given by $b_1 = 0$ and $V$ by $a_1=0$ union
$\lb(a_1^2-b_1^2) + 2b_1 = 0$. The second component of $V$ is a regular curve with tangent direction at the origin is parallel to
the vector $(1,0)$. Therefore, the stratification of the $2$-dimensional section is as
shown in Figure \ref{fig:estratos-cusp} center. Moreover,
$(j^k\phi_{\g_0})^\prime(0) =
(2,0,\dots,0;0,3c_3,\dots,(k+1)c_{k+1})$, which is tangent to the
strata $I$ and $V_1$ and it is transversal to $V_2$ and $L$.
Therefore, the relative position of $j^k\phi_{\g_s}$ with respect to
the strata $C,L,I,V$ is given in Figure \ref{fig:estratos-cusp}.

\begin{figure}[h!]
\centering
\includegraphics[scale = 0.75]{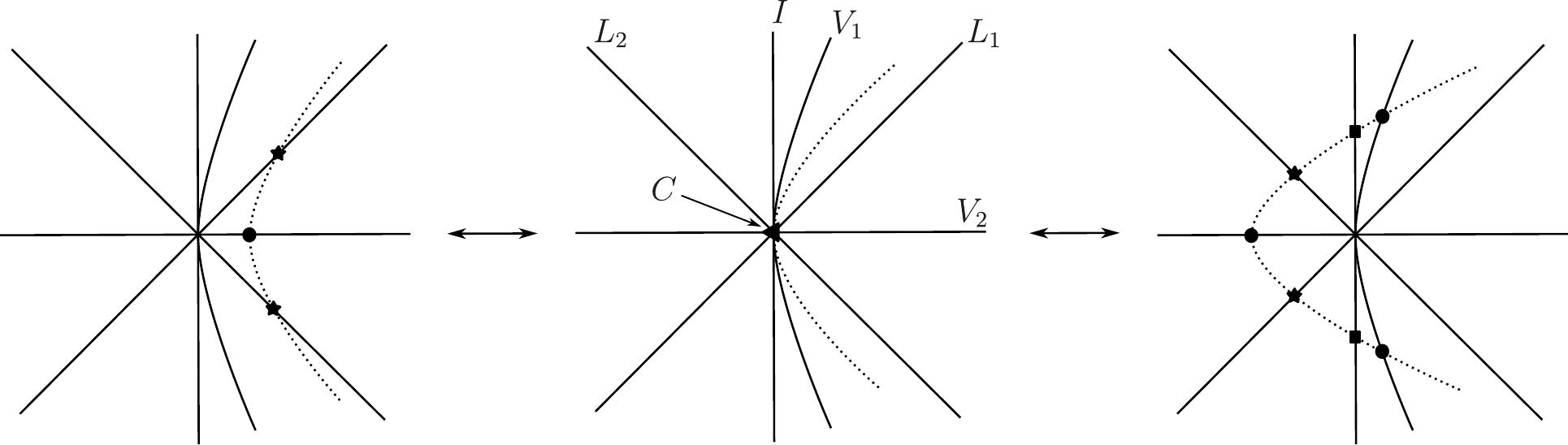}
\caption{\small{Transversal $2$-dimensional section of the
stratification of $J^k(1,2)$ by $C, L, I, V$ and the relative
position of the image of the Monge-Taylor map of $\g_s$ for $s<0$ in
left, $s=0$ in center and $s>0$ in right (case $c_3>
0$).}}\label{fig:estratos-cusp}
\end{figure}

From Figure \ref{fig:estratos-cusp}, we get the number and the
relative position of lightlike points, inflections and vertices,
but we do not obtain their relative position with respect to the
self-intersection points. The self-intersection points of $\g_s$ are
given by $t_1^2 = t_2^2$ and $\b_s(t_1) = \b_s(t_2)$ with $t_1\neq
t_2$, that is, $t_1 = -t_2 = t$ and $\b_s(t) = \b_s(-t)$. Consider
$h(t,s) = \displaystyle\frac{1}{2}(\b_s(t)-\b_s(-t)) = st +
\overline{c_3}(s)t^3 + O_s(t^5) = t(s + \overline{c_3}(s)t^2 +
O_s(t^4)) = t\overline{h}(t,s)$. Hence, the self-intersection points
of $\g_s$ are determined by the zeros of $\overline{h}$. As $c_3\neq
0$, $t^2 = \displaystyle -\frac{1}{c_3}s + O(s^2)$ and the two
solutions of this equation are symmetric and give the values of
$t_1$ and $t_2$ such that $\g_s(t_1) = \g_s(t_2)$. Therefore, $\g_s$ has a self-intersection point in the side
of the transition which contains one vertex and two lightlike
points.

The lightlike points are given by $\displaystyle\b_s^\prime(t) = \pm
2t$. Considering the families $h^{\mp}(t,s) = \b_s^\prime(t)\mp 2t =
s + 2(\overline{c_2}(s)\mp 1)t + 3\overline{c_3}(s)t^2 + O_s(t^3)$,
we obtain that the zeros of $h_s^{\mp}$ and, consequently, the
lightlike points of $\g_s$, are given by $t = \displaystyle \pm
\frac{1}{2}s + O(s^2)$.

Comparing the defining equations of the lightlike and
self-intersection points, we conclude that the lightlike points are
between the self-intersection points.

Note that for $s>0$, the curve $\g_s$ has $2$ inflections ($\kp_s =
0$), $3$ vertices ($\kp_s^\prime = 0$) and $2$ lightlike points
(where $\kp_s$ goes to infinity). Moreover, the signal of $\kp_s$ is the same of $\l \g^{\prime\prime}_s,\g_s^{\prime\perp} \r$. As $t=0$ is between the lightlike
points and the signal of $\l \g^{\prime\prime}_s,\g_s^{\prime\perp} \r$ does not change when we cross a lightlike point, then $\kp_s >0$ between the inflections and $\kp_s<0$ outside.
Hence, the graph of $k_s$ is as in Figure
\ref{fig:graf-curv} right. Therefore, the central vertex is inward
and the others are outward.

For $s = 0$, we have $\kp_0(t) = \displaystyle\frac{6c_3t^2 +
O(t^3)}{|4t^2 -9c_3^2t^4 + O(t^5)|^\frac{3}{2}} =
\frac{3b_3}{4}\frac{1}{|t|}( 1 + O(t))$. Thus, the graph of $\kp_0$ is as in Figure \ref{fig:graf-curv} center.

For $s<0$, the curve $\g_s$ has $1$ vertex, $2$ lightlike points and
no inflections. As $k_s(0)< 0$, then $k_s<0$. Hence, the graph of
$\kp_s$ is as in Figure \ref{fig:graf-curv}
left. We have $\kp_s\kp_s^{\prime\prime}
>0$ so the vertex is inward.

\begin{figure}[h!]
\centering
\includegraphics[scale = 0.8]{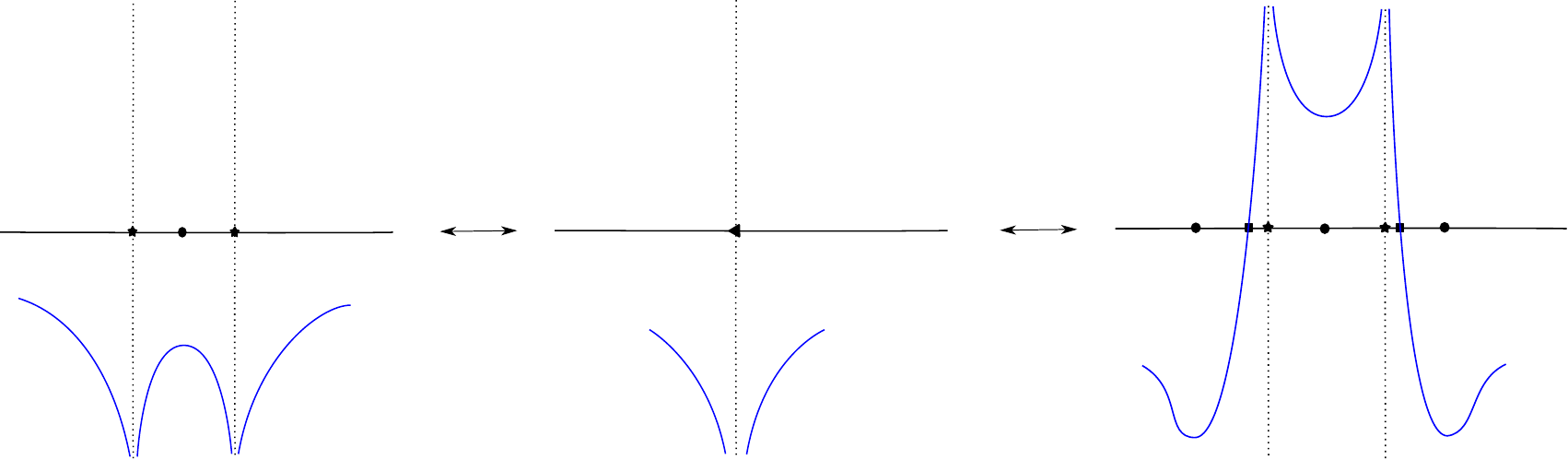}
\caption{\small{Transitions in the graph of the curvature function
of $\g_s$.}}\label{fig:graf-curv}
\end{figure}

From the above analysis and the results in the
previous sections we conclude that the bifurcations in the caustic
of $\g_s$ are as in Figure \ref{fig:def-cusp}.

\text{ (iii)} The results in (ii) depend only on the
fact that the curve has a non-lightlike ordinary cusp
and on the family being \textit{FRS}-generic. The family $\a_u$ satisfies such
conditions, consequently, it is a model for \textit{FRS}-generic
deformations.
\end{proof}

%%%%%%%%%%%%%%%%%%%%%%%%%%%%%%%%%%%%%%%%%%%%%%%%%%%%%%%%%%%%%%%%%%%%%%%%%%%%%%%%%%%%%%%%%%%%%%%%%%%%%%%%%%%%%%%%%%%%%%%%
%%%%%%%%%%%%%%%%%%%%%%%%%%%%%%%%%%%%%%%%%%%%%%%%%%%%%%%%%%%%%%%%%%%%%%%%%%%%%%%%%%%%%%%%%%%%%%%%%%%%%%%%%%%%%%%%%%%%%%%%
\section{\bf Geometric deformations of curves with a lightlike ordinary cusp}
%%%%%%%%%%%%%%%%%%%%%%%%%%%%%%%%%%%%%%%%%%%%%%%%%%%%%%%%%%%%%%%%%%%%%%%%%%%%%%%%%%%%%%%%%%%%%%%%%%%%%%%%%%%%%%%%%%%%%%%%
%%%%%%%%%%%%%%%%%%%%%%%%%%%%%%%%%%%%%%%%%%%%%%%%%%%%%%%%%%%%%%%%%%%%%%%%%%%%%%%%%%%%%%%%%%%%%%%%%%%%%%%%%%%%%%%%%%%%%%%%

Consider $\g$ a plane curve with a lightlike ordinary cusp
singularity at $t_0 = 0$ (i.e., its limiting tangent direction at $t_0$ is lightlike). Here we need to consider multilocal
strata and stratify the
multi-jet space $ _2J^k(1,2)\subset J^k(1,2)\t J^k(1,2)$. We take the Monge-Taylor bi-jet map as
in (\ref{Monge-Taylor:multi}). Denote by
$((a_0,a_1,\dots,a_k;b_0,b_1,\dots,b_k);\linebreak (a_0^\prime,a_1^\prime,\dots,a_k^\prime;b_0^\prime,b_1^\prime,\dots,b_k^\prime))$
the coordinates in $J^k(1,2)\t J^k(1,2)$.

The local strata in $J^k(1,2)$ are viewed product strata in $
_2J^k(1,2)$, consequently, it is enough to compute
them in $J^k(1,2)$. We have:
\begin{itemize}
\item[$LC$:] $a_1 = 0, \ b_1 = 0,\ a_2 - b_2 = 0$
\item[$C$:] $a_1 = 0,\ b_1 = 0$
\item[$I(2)$:] $a_1b_2 - a_2b_1 = 0,\ a_1b_3 - a_3b_1 = 0$
\item[$LI$:] $a_1b_2 - a_2b_1 = 0,\ a_1 \pm b_1 = 0$
\item[$V(2)$:] $(a_1^2-b_1^2)(a_1b_3-a_3b_1) +
2(b_1b_2-a_1a_2)(a_1b_2-a_2b_1) = 0,\
(-a_1^2+b_1^2)(2(a_1b_4-a_4b_1) + a_2b_3 - a_3b_2)-
(a_1b_3-a_3b_1)(b_1b_2-a_1a_2)-(a_1b_2-a_2b_1)(3(b_1b_3-a_1a_3)+2(b_2^2-a_2^2))
= 0$
\item[$IT$:] $a_0 - a_0^\prime = 0,\ b_0 - b_0^\prime = 0,\
a_1b_2 - a_2b_1 = 0$
\item[$LT$:] $a_0 - a_0^\prime = 0,\ b_0 -
b_0^\prime = 0,\ a_1 \pm b_1 = 0$
\item[$Tc$:] $a_0 - a_0^\prime = 0,\ b_0 - b_0^\prime = 0,\
a_1b_1^\prime - a_1^\prime b_1 = 0$
\item[$VT$:] $a_0 - a_0^\prime = 0,\ b_0 - b_0^\prime = 0,\
(a_1^2-b_1^2)(a_1b_3-a_3b_1) + 2(b_1b_2-a_1a_2)(a_1b_2-a_2b_1) = 0.$
\end{itemize}

We take $\g(t) =
(t^2,t^2 + c_3t^3 + c_4t^4 + O(t^5))$, with $c_3\neq 0$.

Let $\g_s$ be a $2$-parameter deformation of $\g$ which we can take in the form
\begin{equation}\label{eq:def:cusp:light}
\g_s(t) = (t^2,\overline{c_1}(s)t + (\overline{c_2}(s)+1)t^2 +
\overline{c_3}(s)t^3 + \overline{c_4}(s)t^4 + O_s(t^5)),
\end{equation}
with $\overline{c_1}(0) = \overline{c_2}(0) = 0$, $\overline{c_3}(0)
= c_3\neq 0$ and $\overline{c_k}(0) = c_k$, for $k\> 4$.

\begin{theorem}\label{prop:trans:cusp:light}
Consider $\g_s$ a $2$-parameter deformation of $\g$ as in
{\rm(\ref{eq:def:cusp:light})}. Then the family of Monge-Taylor
bi-jet maps $ _2j^k\Phi$ is transverse to the strata $LC$, $C$,
$I(2)$, $LI$, $V(2)$, $IT$, $LT$, $VT$ and $Tc$ if, and only if,
$\frac{\partial\overline{c_1}}{\partial s_1}(0,0)\neq 0$ and
$\frac{\partial \overline{c_1}}{\partial s_1}(0,0)\frac{\partial
\overline{c_2}}{\partial s_2}(0,0) - \frac{\partial
\overline{c_1}}{\partial s_2}(0,0)\frac{\partial
\overline{c_2}}{\partial s_1}(0,0)\neq 0$.
%$$\left|\begin{array}{cc}
%    \frac{\partial \overline{c_1}}{\partial s_1}(0,0) & \frac{\partial \overline{c_1}}{\partial s_2}(0,0) \\
%    \frac{\partial \overline{c_2}}{\partial s_1}(0,0) & \frac{\partial \overline{c_2}}{\partial s_2}(0,0)
%\end{array}\right|\neq 0.$$

\end{theorem}
\begin{proof}
The proof follows by standard lengthy calculations and is omitted (see \cite{Tese} for details). 
\end{proof}

The condition $\frac{\partial c_1}{\partial s_1}(0,0)\neq 0$ is
exactly the condition to $\g_s$ be an $\mathcal{A}_e$-versal
deformation of $\g$.

\begin{corollary}\label{def:FRLS-generic}
Let $\g$ be a lightlike ordinary cusp and $\g_s$ be a $2$-parameter
deformation of $\g_0 = \g$ as in {\rm(\ref{eq:def:cusp:light})}. Then $\g_s$ is an \textit{FRLS}-generic deformation if, and only if, $\g_s$ is
an $\mathcal{A}_e$-versal deformation of $\g_0$ and $\frac{\partial
\overline{c_1}}{\partial s_1}(0,0)\frac{\partial
\overline{c_2}}{\partial s_2}(0,0) - \frac{\partial
\overline{c_1}}{\partial s_2}(0,0)\frac{\partial
\overline{c_2}}{\partial s_1}(0,0)\neq 0$.
%$$\left|\begin{array}{cc}
%    \frac{\partial \overline{c_1}}{\partial s_1}(0,0) & \frac{\partial \overline{c_1}}{\partial s_2}(0,0) \\
%    \frac{\partial \overline{c_2}}{\partial s_1}(0,0) & \frac{\partial \overline{c_2}}{\partial s_2}(0,0)
%\end{array}\right|\neq 0.$$
\end{corollary}

By Corollary \ref{def:FRLS-generic}, if $\g_s$ is a $2$-parameter deformation
as in (\ref{eq:def:cusp:light}) and is \textit{FRLS}-generic, we can
take $\g_s(t) = (t^2,\b_s(t))$, with
\begin{equation}\label{eq:cusp-light}
\b_s(t) = s_1t + (s_2 + 1)t^2 + \overline{c_3}(s)t^3 +
\overline{c_4}(s)t^4 +  O_s(t^5).
\end{equation}

\begin{theorem}\label{teo:est-cusp-light}
Let $\g_s$ be an \textit{FRLS}-generic $2$-parameter deformation of a
lightlike ordinary cusp $\g$ as in {\rm(\ref{eq:cusp-light})}. Then,
the stratification of the parameter space $s = (s_1,s_2)\in\R^2,0$
of $\g_s$ (see Figure \ref{fig:def-cusp-light}, center) consists
of the origin ($LC$) and the following curves:
\begin{itemize}
\item[$C$:] $s_1 = 0$
\item[$LI$:] $(s_1,s_2) = (3c_3t^2+O(t^3),-3c_3t+O(t^2))$
\item[$V(2)$:]  $\left\{\begin{array}{l}
     (s_1,s_2) = \displaystyle (-\frac{2\sqrt{5} - 5}{25c_3}t^2 + O(t^3), t + O(t^2)) \\
     (s_1,s_2) = \displaystyle (\frac{2\sqrt{5} + 5}{25c_3}t^2 + O(t^3), t + O(t^2))
   \end{array}\right.$
\item[$LT$:]   $(s_1,s_2) = (-c_3t^2+c_3c_{301}t^3+O(t^4),-c_3t+(c_3c_{301}-2c_4)t^2+O(t^3))$
\item[$VT$:]   $(s_1,s_2) = (-c_3t^2+c_3c_{301}t^3+O(t^4),-c_3t+(c_3c_{301}+2c_3-2c_4)t^2+O(t^3))$
\end{itemize}
where $c_{301}$ is the coefficient of $s_2$ in $\overline{c_3}(s)$. The strata $I(2), IT$ and $Tc$ are empty.
\end{theorem}
\begin{proof}
The stratification of the parameter space of $\g_s$ is the projection to that space of the pre-image by the Monge-Taylor map of the
stratification of the jet space by phenomena of codimension
$\< 2$. We present here the calculations for some stratum, the others follows in a similar way.

Consider the stratum $LI$ given by $a_1\pm b_1 = 0$ and
$a_1b_2 - a_2b_1 = 0$, that is, $b_1 = \mp a_1$ and $a_1(b_2 \pm
a_2) = 0$. If $a_1 = 0$ then $b_1 = 0$ and, consequently, we are on the
cusp stratum. Hence, $b_2 \pm a_2 = 0$ so $LI = \{a_1+b_1 = 0, a_2+b_2 = 0\}\cup\{a_1-b_1 = 0,
a_2-b_2 = 0\}$. Note that $_2j^k\Phi$ does not intersect the first component of $LI$. Taking the
pre-image of the second component by $_2j^k\Phi$ we get $\{2t -
\b_s^\prime(t) = 0, 1 - \frac{1}{2}\b_s^{\prime\prime}(t) = 0\}$ and this gives $LI$ parametrized by $(s_1,s_2) = (3c_3t^2 + O(t^3),-3c_3t + O(t^2))$.

For the stratum $LT$ we take its pre-image by $_2 j^k\Phi$ and we
get $t_1^2-t_2^2 = 0$, $\b_s(t_1)-\b_s(t_2) = 0$ and
$2t_1 \pm \b_s^\prime(t_1) = 0$, that is,  $t_1 = -t_2 = t$,
$\b_s(t)-\b_s(-t) = 0$ and $2t \pm \b_s^\prime(t) = 0$. 
A calculation shows now that $s_1 = -c_3t^2 + c_3c_{301}t^3+O(t^4)$ and $s_2 = -c_3t+(c_3c_{301}-2c_4)t^2+O(t^3)$.

The stratum $V(2)$ is given by the zeros of multiplicity $2$ of the
family $G(t,s)$ obtained by taking the numerator of $\kp_s^\prime(t)$. As
$G$ is a deformation of an $A_3$ singularity, $G$ is $\Rl$-induced
by $F(t, u_0,u_1,u_2) = u_0 + u_1t + u_2t^2 + t^4$, that is, we can
write $G(t,s_1,s_2) = a_0(s_1,s_2) + a_1(s_1,s_2)t + a_2(s_1,s_2)t^2 + t^4,$ for some smooth map-germs $a_0,a_1,a_2$. Thus, $V(2)$ is the
zero set of the resultant $R$ of $G$ and $\frac{\partial G}{\partial
t}$ with respect to $t$. We have
$R(s_1,s_2) =
16a_0a_2^4-4a_1^2a_2^3-128a_0^2a_2^2+144a_0a_1^2a_2-27a_1^4+256a_0^3$.

Analysing $R$, we find that $V(2)$ consists of germs of three regular and tangent curves
parametrized by
\begin{eqnarray*}
(s_1,s_2) &=& \displaystyle (-\frac{2\sqrt{5} - 5}{25c_3}t^2 + O(t^3), t + O(t^2)), \\
(s_1,s_2) &=& \displaystyle (\frac{2\sqrt{5} + 5}{25c_3}t^2 + O(t^3), t + O(t^2)),\\
(s_1,s_2) &=& \displaystyle (\frac{1}{3c_3}t^2 + O(t^3), t +
O(t^2)).
\end{eqnarray*}
By Theorem \ref{teo:IL-V}, $LI\subset V(2)$ and the last component
above represents $LI$. Therefore, the stratum $V(2)$ consists only of the first above two curves.
\end{proof}

\begin{theorem}\label{teo:LC} Let $\g$ be a germ of a lightlike ordinary cusp at $t_0$.
\begin{itemize}
\item[\text{(i)}] The caustic of $\g$ is the union of the tangent line
of $\g$ at $t_0$ and a curve with a lightlike ordinary cusp
singularity at $\g(t_0)$ as shown in Figure
\ref{fig:def-cusp-light-ev} $\textcircled{1}$.

\item[\text{(ii)}] The bifurcations of an \textit{FRLS}-generic $2$-parameter
family $\g_s$ with $\g_0=\g$ are as shown in Figure
\ref{fig:def-cusp-light}.

\item[\text{ (iii)}] For an \textit{FRLS}-generic $2$-parameter family $\g_s$
with $\g_0=\g$, the bifurcations in the caustic are as shown in
Figure \ref{fig:def-cusp-light-ev}.

\item[\text{ (iv)}] Any \textit{FRLS}-generic $2$-parameter family of $\g$ is \textit{FRLS}-equivalent to the model family $\a_u(t) = (t^2, u_1t + (1+u_2)t^2 + t^3)$.
\end{itemize}
\end{theorem}

\begin{figure}[h!]
\centering
\includegraphics[scale = 0.65]{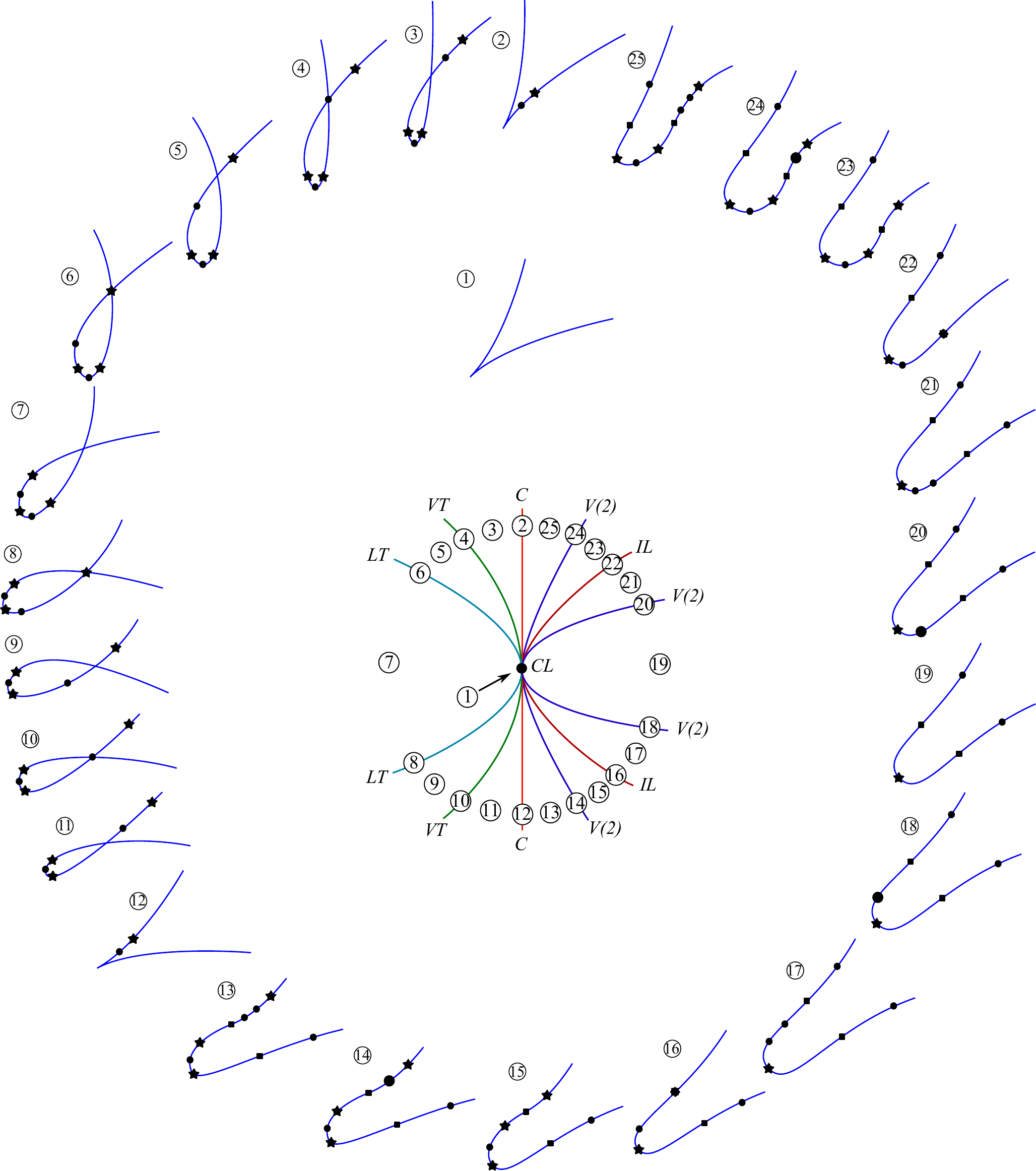}
\caption{\small{FRLS-generic deformation of a curve with a lightlike ordinary
cusp.}}\label{fig:def-cusp-light}
\end{figure}

\begin{figure}[h!]
\centering
\includegraphics[scale = 0.65]{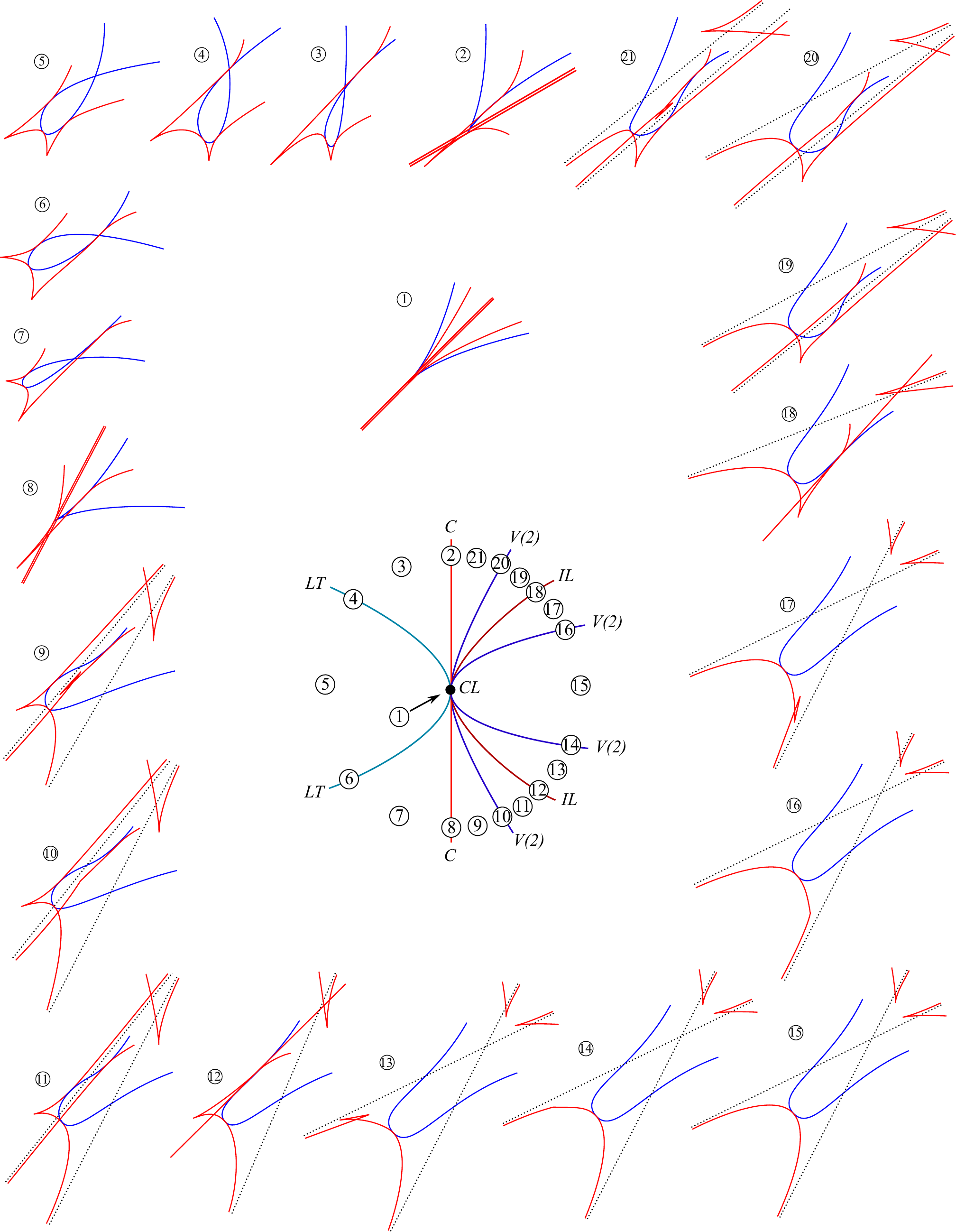}
\caption{\small{\textit{FRLS}-generic deformations of the caustic of a curve with a
lightlike ordinary cusp.}}\label{fig:def-cusp-light-ev}
\end{figure}

\begin{proof}

(i) We take $\g(t) = (t^2,\b(t)) = (t^2,t^2+c_3t^3+c_4t^4+O(t^5))$ at $t_0 = 0$.
Thus, $d_u(t) = -(t^2-u_1)^2+(\b(t)-u_2)^2$. Equations $d_u^{\prime\prime} = 0$ and $d_u^{\prime}
= 0$ give
$$t^2(4t+\frac{(\b^{\prime\prime}(t)t-\b^\prime(t))}{t^2}u_2-\frac{(\b^{\prime\prime}(t)t-\b^\prime(t))}{t^2}\b(t)-\frac{\b^\prime(t)^2t}{t^2})
= 0.$$ As $t = 0$ is a root of the above equation, the line $u_1 =
u_2$ is part of the caustic of $\g$.

For $t\neq 0$, we get $\displaystyle u_2 =
\frac{-4t^3+(\b^{\prime\prime}(t)t-\b^\prime(t))\b(t)+\b^\prime(t)^2t}{\b^{\prime\prime}(t)t-\b^\prime(t)}$.
Therefore, the caustic of $\g$ is the union of the lightlike line
$u_1 = u_2$ and the curve parametrized by $\displaystyle(5t^2 +
(9c_3-\frac{16c_4}{3c_3})t^3+O(t^4),5t^2+(4c_3-\frac{16c_4}{3c_3})t^3+O(t^4))$.
Comparing this lightlike ordinary cusp with $\g$ we conclude that
the relative position of the two curves is as shown in Figure
\ref{fig:def-cusp-light-ev} $\textcircled{1}$.

(ii) The lightlike points are given by $h^\pm(t,s) = \b_s^\prime(t)\pm 2t =
0$. The
family $h^+$ is regular with $h^+(0,0,0) = 0$, then $h^+$ has one
zero for any $s$. The family $h^-$ is an $\mathcal{R}$-versal
deformation of an $A_1$ singularity and, consequently, it has at
most 2 zeros. Therefore, $\g_s$ has at most 3 lightlike points.

For inflections, we use the numerator of $\kp_s$ given by $G(t,s) =
-2s_1 + 6\overline{c_3}(s)t^2 + O_s(t^3)$ which is a deformation of
an $A_1$ singularity. Thus $\g_s$ has at most 2 inflections.

For vertices, we note that the numerator of $\kp_s^\prime$ is given
by $H(t,s) =
-8t^3\b_s^{\prime\prime\prime}(t)+2t\b_s^\prime(t)^2\b_s^{\prime\prime\prime}(t)
-24t\b_s^\prime(t)+24t^2\b_s^{\prime\prime}(t)+6\b_s^\prime(t)^2\b_s^{\prime\prime}(t)-6t\b_s^\prime(t)\b_s^{\prime\prime}(t)^2$
which is a deformation of an $A_3$ singularity. Then $\g_s$ has at
most 4 vertices.

To obtain the configuration of the lightlike points, inflections and
vertices in $\g_s$ we use the results obtained in the previous
sections and the fact that changes occur only when we cross a curve
in the stratification of the parameter space given in Theorem
\ref{teo:est-cusp-light}. 

Firstly, we find the configuration on the
cusp stratum, that is, when $s_1 = 0$.
In this case, $h^+(t,0,s_2) = t(4 + 2s_2 + 3\overline{c_3}(s_2)t +
O_{s_2}(t^2))$ and $h^-(t,0,s_2) = t(2s_2 + 3\overline{c_3}(s_2)t +
O_{s_2}(t^2))$. Therefore, using $h^-$, we
get $t = -\frac{2}{3c_3}s_2 + O(s_2^2)$ as the unique lightlike point
of $\g_{s}$ on the cusp stratum.

For inflections, we have $G(t,0,s_2) = t^2(6\overline{c_3}(s_2) +
O_{s_2}(t))$ which implies that the two inflections of $\g_s$ are
concentrated at the cusp point.

For vertices, we have $H(t,0,s_2) = t^3(-48c_3s_2-180c_3^2t +
O_2(t,s_2))$, that is, $\g_{s}$ has 3 vertices concentrated at the
cusp point and another given by $t = - \frac{4}{15c_3}s_2+O(s_2^2)$.

The points are positioned in the following order L - V - C when $s_2>0$ and C - V - L
when $s_2<0$. Moreover, the vertex is inward, since
$sgn(\kp_{s_2}(t)) = sgn(G(t,s_2)) = sgn(c_3)$, for all $t$
sufficiently close to $0$.

Going around the origin in the $(s_1,s_2)$-space and using the results of the sections $5, 6$ and $7$ we obtain the configuration of $\g_s$ in each stratum as
shown in Figure \ref{fig:def-cusp-light}.

(iii) The statement follows from (ii) and the results in sections $5,6$ and $7$.

(iv) From Theorem \ref{teo:est-cusp-light}, the stratifications
of the parameter spaces of $\g_s$ and $\a_u$ are homeomorphic.
Moreover, the calculations in (ii) and (iii) depend only on the
fact that the curve has a lightlike ordinary cusp and the family is
\textit{FRLS}-generic, conditions that $\a_u$ satisfies.
\end{proof}

%%%%%%%%%%%%%%%%%%%%%%%%%%%%%%%%%%%%%%%%%%%%%%%%%%%%%%%%%%%%%%%%%%%%%%%%%%%%%%%%%%%%%%%%%%%%%%%%%%%%%%%%%%%%%%%%%%%%%%%%
%%%%%%%%%%%%%%%%%%%%%%%%%%%%%%%%%%%%%%%%%%%%%%%%%%%%%%%%%%%%%%%%%%%%%%%%%%%%%%%%%%%%%%%%%%%%%%%%%%%%%%%%%%%%%%%%%%%%%%%%
\section{\bf Geometric deformations of curves with a ramphoid cusp singularity}
%%%%%%%%%%%%%%%%%%%%%%%%%%%%%%%%%%%%%%%%%%%%%%%%%%%%%%%%%%%%%%%%%%%%%%%%%%%%%%%%%%%%%%%%%%%%%%%%%%%%%%%%%%%%%%%%%%%%%%%%
%%%%%%%%%%%%%%%%%%%%%%%%%%%%%%%%%%%%%%%%%%%%%%%%%%%%%%%%%%%%%%%%%%%%%%%%%%%%%%%%%%%%%%%%%%%%%%%%%%%%%%%%%%%%%%%%%%%%%%%%

Consider $\g$ a curve with a non-lightlike ramphoid cusp singularity
$\mathcal{A}_h$-equivalent to $(t^2,t^4+t^5)$ (see \cite{nunu} for
the $\mathcal{A}_h$-equivalence). The strata of interest are:
\begin{itemize}
\item[$RC$:] $a_1 = 0, \ b_1 = 0, \ a_2b_3 - a_3b_2 = 0$
\item[$C$:] $a_1 = 0, \  b_1 = 0$
\item[$I(2)$:] $a_1b_2 - a_2b_1 = 0, \ a_1b_3 - a_3b_1 = 0$
\item[$V(2)$:] $(a_1^2-b_1^2)(a_1b_3 - a_3b_1) + 2(b_1b_2 -
a_1a_2)(a_1b_2 - a_2b_1) = 0,\
2(a_1b_2-a_2b_1)[-(-a_1^2+b_1^2)(-a_2^2+b_2^2)+5(b_1b_2-a_1a_2)^2]
+(-a_1^2+b_1^2)[(4a_2b_3+4a_3b_2)a_1^2-9(a_2a_3+b_2b_3)a_1b_1
+(4a_2b_3+5a_3b_2)b_1^2] + 2(-a_1^2+b_1^2)^2(a_1b_4-a_4b_1) = 0$
\item[$LI$:] $a_1\pm b_1 = 0, \ a_1b_2 - a_2b_1 = 0$
\item[$IT$:] $a_0-a_0^\prime = 0, \ b_0-b_0^\prime = 0, \ a_1b_2 -
a_2b_1 = 0$
\item[$VT$:] $a_0-a_0^\prime = 0, \ b_0-b_0^\prime = 0, \
(a_1^2-b_1^2)(a_1b_3 - a_3b_1) + 2(b_1b_2 - a_1a_2)(a_1b_2 - a_2b_1)
= 0$
\item[$LT$:] $a_0-a_0^\prime = 0, \ b_0-b_0^\prime = 0, \ a_1\pm b_1 = 0$
\item[$Tc$:] $a_0-a_0^\prime = 0, \ b_0-b_0^\prime = 0, \ a_1b_1^\prime
-a_1^\prime b_1 = 0.$
\end{itemize}

%\begin{itemize}
%\item[] $RC: \ \ \left\{\begin{array}{l}
%a_1 = 0\\
%b_1 = 0\\
%a_2b_3 - a_3b_2 = 0
%\end{array}\right.$ \ \ \ \ $C: \ \ \left\{\begin{array}{l}
%a_1 = 0\\
%b_1 = 0
%\end{array}\right.$ \ \ \ \ \ $I(2): \ \ \left\{\begin{array}{l}
%a_1b_2 - a_2b_1 = 0\\
%a_1b_3 - a_3b_1 = 0
%\end{array}\right.$
%\item[] $V(2): \ \ \left\{\begin{array}{l} (a_1^2-b_1^2)(a_1b_3 - a_3b_1)
%+ 2(b_1b_2 - a_1a_2)(a_1b_2 - a_2b_1) = 0\\
%2(a_1b_2-a_2b_1)[-(-a_1^2+b_1^2)(-a_2^2+b_2^2)+5(b_1b_2-a_1a_2)^2] +\\
%+(-a_1^2+b_1^2)[(4a_2b_3+4a_3b_2)a_1^2-9(a_2a_3+b_2b_3)a_1b_1 +\\
%+(4a_2b_3+5a_3b_2)b_1^2] + 2(-a_1^2+b_1^2)^2(a_1b_4-a_4b_1) = 0
%\end{array}\right.$
%\item[] $LI: \ \ \left\{\begin{array}{l}
%a_1\pm b_1 = 0\\
%a_1b_2 - a_2b_1 = 0
%\end{array}\right.$ \ \ \ \ \ \ \ \ \ \ \ \ $IT: \ \ \left\{\begin{array}{l}
%a_0-a_0^\prime = 0\\
%b_0-b_0^\prime = 0\\
%a_1b_2 - a_2b_1 = 0
%\end{array}\right.$
%\item[] $VT: \ \ \left\{\begin{array}{l}
%a_0-a_0^\prime = 0\\
%b_0-b_0^\prime = 0\\
%(a_1^2-b_1^2)(a_1b_3 - a_3b_1) + 2(b_1b_2 - a_1a_2)(a_1b_2 - a_2b_1)
%= 0
%\end{array}\right.$
%\item[] $LT: \ \ \left\{\begin{array}{l}
%a_0-a_0^\prime = 0\\
%b_0-b_0^\prime = 0\\
%a_1\pm b_1 = 0
%\end{array}\right.$ \ \ \ \ \ \ \ \ \ \ \ \ \ \ \ \ \ $Tc: \ \ \left\{\begin{array}{l}
%a_0-a_0^\prime = 0\\
%b_0-b_0^\prime = 0\\
%a_1b_1^\prime -a_1^\prime b_1 = 0
%\end{array}\right.$
%\end{itemize}
%where $RC$ denote the stratum of ramphoid cusp.

Suppose the limiting tangent direction of $\g$ at the ramphoid cusp is timelike, the spacelike case is similar. We
take $\g(t) = (t^2,\b(t)) = (t^2,c_4t^4+c_5t^5+O(t^6))$, with
$c_4,c_5\neq 0$. Let $\g_s$ be a $2$-parameter deformation of $\g$
which can be taken in the form $\g_s(t) = (t^2,\b_s(t))$, where
\begin{equation}\label{eq:cusp-ramph}
 \b_s(t) = \overline{c_1}(s)t+\overline{c_2}(s)t^2+\overline{c_3}(s)t^3+\overline{c_4}(s)t^4+\overline{c_5}(s)t^5+O_s(t^6)
\end{equation}
with $\overline{c_1}(0) = \overline{c_2}(0) = \overline{c_3}(0) = 0$
and $\overline{c_k}(0) = c_k$, with $k\> 4$.

The strata $RC, C$ and $Tc$ are all the strata in $ _2J^k(1,2)$ of
codimension $\< 3$ that come from the $\mathcal{A}$-equivalence.
Hence, from the Mather's Theorem, to ensure the transversality of
$_2j^k\Phi$ with respect to such strata we need the
$\mathcal{A}_e$-versality of $\g_s$. We can show that $\g_s$ as in
(\ref{eq:cusp-ramph}) is an $\mathcal{A}_e$-versal deformation of
$\g$ if, and only if, $\frac{\partial \overline{c_1}}{\partial
s_1}(0)\frac{\partial \overline{c_3}}{\partial s_2}(0) -
\frac{\partial \overline{c_1}}{\partial s_2}(0)\frac{\partial
\overline{c_3}}{\partial s_1}(0)$. Therefore, supposing $\g_s$ an
$\mathcal{A}_e$-versal deformation of $\g$, we can take $\b_s(t) =
s_2t+\overline{c_2}(s)t^2+s_1t^3+\overline{c_4}(s)t^4+\overline{c_5}(s)t^5+O_s(t^6)$.

\begin{theorem}\label{teo:est-cusp-ramp}
Let $\g_s$ be an $\mathcal{A}_e$-versal $2$-parameter deformation of
a ramphoid cusp $\g$ $\mathcal{A}_h$-equivalent to $(t^2,t^4+t^5)$.
Then, the family of Monge-Taylor bi-jet maps $_2j^k\Phi$ is
transverse to the above strata and the stratification of the
parameters space of $\g_s$ (given in Figure \ref{fig:def-cusp-ramph}
center) consists of the origin ($RC$) and the following curves:
\begin{itemize}
\item[] $C$: \ $s_2 = 0$
\item[] $I(2)$: \ $(s_1,s_2) = (-4c_4t+O(t^2),-4c_4t^3+O(t^4))$
\item[] $V(2)$: \ $(s_1,s_2) = (10c_5t^2+40c_6t^3+O(t^4),5c_5t^4+24c_6t^5+O(t^6))$
\item[] $IT$: \ $(s_1,s_2) = (-2c_4t+O(t^2),2c_4t^3+O(t^4))$
\item[] $VT$: \ $(s_1,s_2) = (2c_5t^2+8c_6t^3+O(t^4),-3c_5t^4-8c_6t^5+O(t^6))$
\item[] $Tc$: \ $(s_1,s_2) = (-2c_5t^2+O(t^3),c_5t^4+O(t^5))$.
\end{itemize}
The strata $LI$ and $LT$ are empty.
\end{theorem}
\begin{proof}
It follows similarly to the proof of Theorem \ref{prop:trans:cusp:light} and \ref{teo:est-cusp-light} (see \cite{Tese} for details).
\end{proof}

From Theorem \ref{teo:est-cusp-ramp}, the strata $V(2)$ and $VT$
have a ramphoid cusp singularity if $c_6\neq 0$. This is a new
condition that will be required for the concept of \textit{FRS}-genericity.

\begin{corollary}\label{def:ramphoid}
Consider the curve $\g(t) = (t^2,\b(t)) =
(t^2,c_4t^4+c_5t^5+c_6t^6+O(t^7))$ with $c_4\neq 0,c_5\neq 0$ and
$c_6\neq 0$. Then a $2$-parameter deformation $\g_s$ of $\g$
is \textit{FRS}-generic if, and only if, $\g_s$ is an $\mathcal{A}_e$-versal
deformation of the ramphoid cusp singularity of $\g$ at $t = 0$.
\end{corollary}

\begin{theorem} Let $\g$ be a germ of a non-lightlike ramphoid cusp $\mathcal{A}_h$-equivalent to $(t^2,t^4+t^5)$ at
$t_0$.
\begin{itemize}
\item[\text{(i)}] The caustic of $\g$ is the union of the tangent line of $\g$ at $t_0$
and a smooth curve with an ordinary inflection as shown in Figure
\ref{fig:def-cusp-ramph-evoluta} $\textcircled{1}$.

\item[\text{(ii)}] The bifurcations of an \textit{FRS}-generic $2$-parameter family
$\g_s$ with $\g_0=\g$ are as shown in Figure
\ref{fig:def-cusp-ramph}.

\item[\text{(iii)}] For an \textit{FRS}-generic $2$-parameter family $\g_s$ with
$\g_0=\g$, the bifurcations in the caustic are as shown in Figure
\ref{fig:def-cusp-ramph-evoluta}.

\item[\text{(iv)}] Any \textit{FRS}-generic $2$-parameter family of $\g$ is \textit{FRS}-equivalent to
the model family $\a_u(t) = (t^2, u_2t + u_1t^3 + t^4 + t^5 + t^6)$.
\end{itemize}
\end{theorem}

\begin{figure}[h!]
\centering
\includegraphics[scale = 0.75]{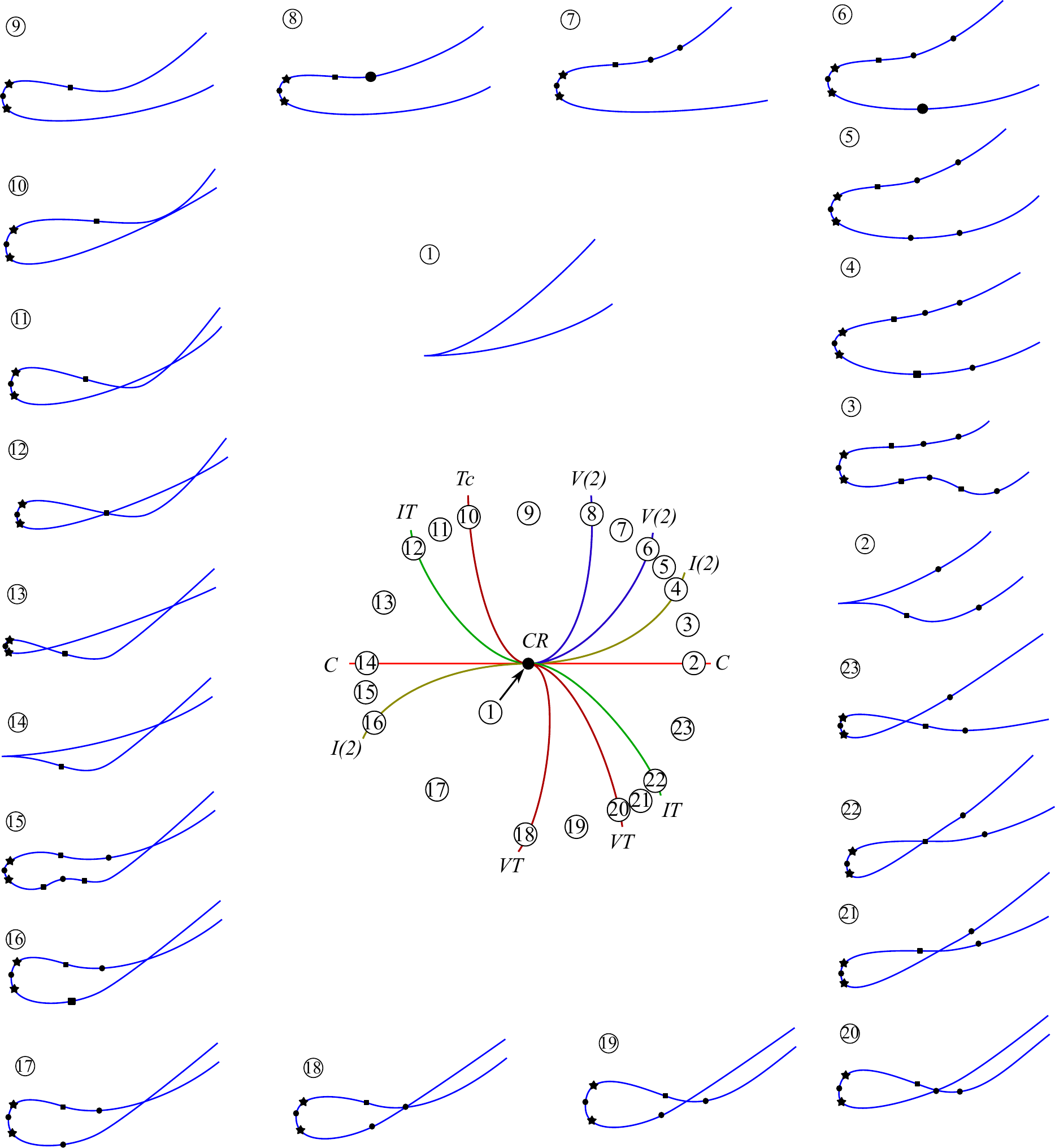}
\caption{\small{\textit{FRS}-generic deformations of a curve with ramphoid
cusp.}}\label{fig:def-cusp-ramph}
\end{figure}

\begin{figure}[h!]
\centering
\includegraphics[scale = 0.65]{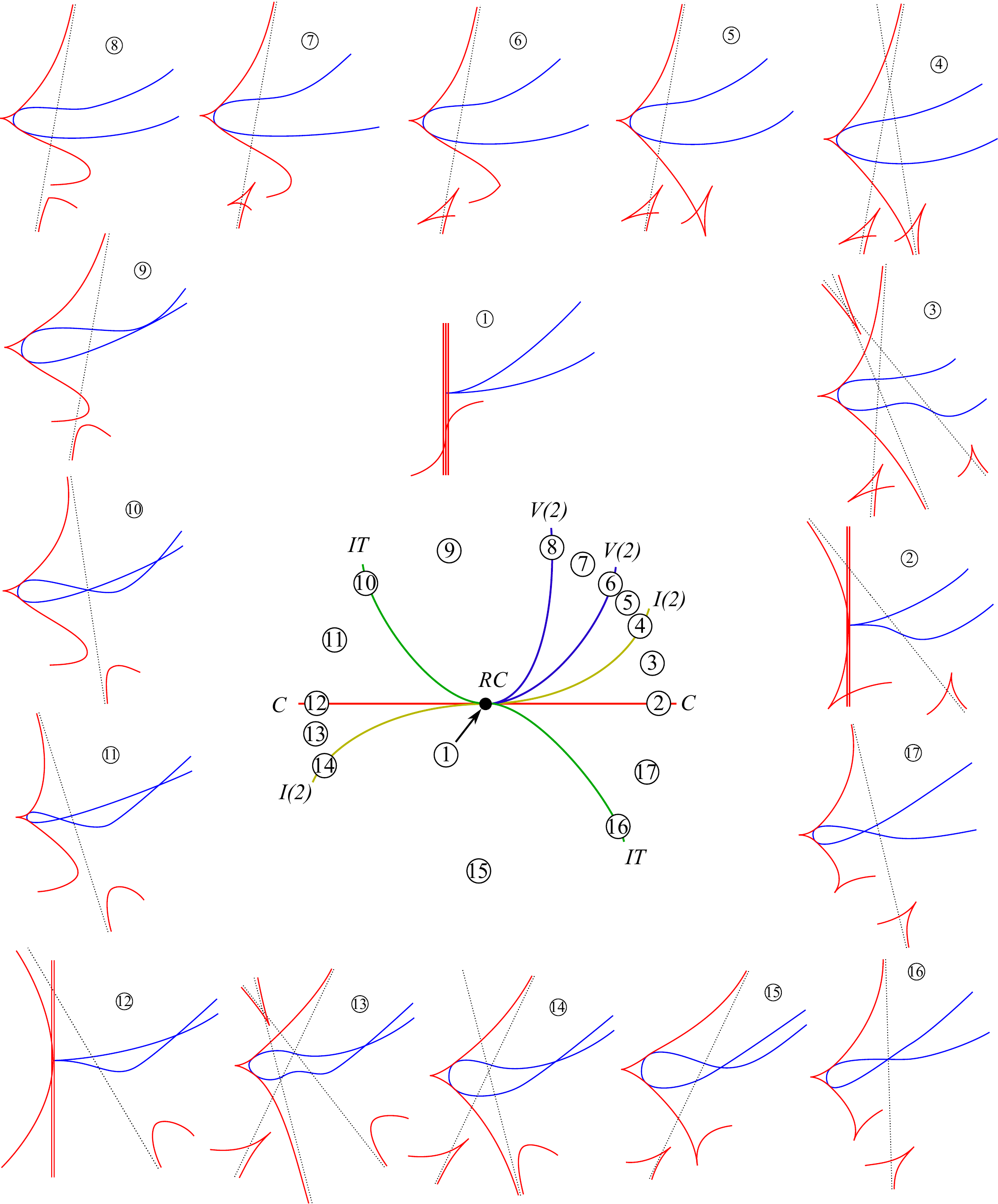}
\caption{\small{\textit{FRS}-generic deformations of the caustic of a curve with a
ramphoid cusp.}}\label{fig:def-cusp-ramph-evoluta}
\end{figure}

\begin{proof}

(i) The proof is similar to that of item (i) of the Theorem
\ref{teo:LC}. We obtain that the caustic of $\g$ is the union of the
line $u_1 = 0$ and a curve given by
$(\displaystyle\frac{5c_5}{8c_4}t^3+O(t^4),-\frac{1}{2c_4}+\frac{15c_5}{16c_4^2}t+O(t^2))$
as shown in Figure \ref{fig:def-cusp-ramph-evoluta}
$\textcircled{1}$.

(ii) The lightlike points of $\g_s$ are given by
$h^\pm(t,s) =\b_s^\prime(t)\pm 2t = 0$, which are regular. Note that $h^+$ and $h^-$
have only one zero each for each fixed $s\in\R^2,0$, given by $t_1
= -\frac{1}{2}s_2+O_2(s_1,s_2)$ and $t_2 =
\frac{1}{2}s_2+O_2(s_1,s_2)$, respectively. Therefore, $\g_s$ has
two lightlike points for all $(s_1,s_2)\in\R^2,0$, with $s_2\neq 0$.

For inflections we consider the numerator of $\kp_s$ given by $G(t,s) =
-2s_2 + 6s_1t^2+8\overline{c_4}(s)t^3 + 15\overline{c_5}(s)t^4 +
24\overline{c_6}(s)t^5 + O_s(t^6)$ which is a deformation of an
$A_2$ singularity. Thus $\g_s$ can have at most $3$ inflections.

For vertices we consider the numerator of $\kp_s^\prime$ given by
$H(t,s) =
2t\b_s^{\prime\prime\prime}(t)(-4t^2+\b_s^\prime(t)^2)-3(-2\b_s^\prime(t)
+
2t\b_s^{\prime\prime}(t))(-4t+\b_s^\prime(t)\b_s^{\prime\prime}(t))$
which is a deformation of an $A_4$ singularity. Then $\g_s$ can have
at most $5$ vertices.

To obtain the relative position of lightlike points, inflections and
vertices in $\g_s$ we use the results in sections $5$ and $7$ and the
fact that changes only occur when we cross a stratum in the parameters space in Theorem
\ref{teo:est-cusp-ramp}. We start by considering the configuration
of $\g_s$ on the cusp stratum, that is, when $s_2 = 0$.

In this case $h^+(t,s_1,0) = t[2 + 2\overline{c_2}(s_1) + 3s_1t +
4\overline{c_4}(s_1)t^2+O_{s_1}(t^3)]$ and $h^-(t,s_1,0) =
t[-2 + 2\overline{c_2}(s_1) + 3s_1t +
4\overline{c_4}(s_1)t^2+O_{s_1}(t^3)]$. Thus, $\g_{s}$ does not have
lightlike points outside the cusp point on the cusp stratum.

We have $G(t,s_1,0) = t^2( 6s_1 + 8\overline{c_4}(s_1)t +
15\overline{c_5}(s_1)t^2 + O_s(t^3))$, that is, $\g_{s}$ has two
inflections concentrated at the cusp point and another one given by
$t_0 = \displaystyle-\frac{3}{8c_4}s_1 + O(s_1^2)$.

For vertices, we have $H(t,s_1,0) =
12t^3(2s_1+O(s_1)s_1-10c_5t^2+O(s_1)t^2+O_{s_1}(t^3))$ which implies
that the vertices of $\g_{s}$ are given by $s_1 = 5c_5t^2+O(t^3)$.
Hence, we have $2$ vertices for $s_1>0$ and none for $s_1<0$.
Moreover, as $c_4$ and $c_5$ are fixed (and suppose, without loss of generality, to be positives), then
$\displaystyle\frac{3s_1}{8c_4} < \sqrt{\frac{s_1}{5c_5}}$, which
implies that the relative position of the relevant points is V-I-C-V for $s_1>0$ and C-I
for $s_1<0$. The first vertex is outward and the second is inward.
In fact, the curvature function is given by
$\displaystyle\kp_{s_1}(t) = \frac{6s_1t^2 +
O_{s_1}(t^3)}{|4t^2(-1+\overline{c_2}(s_1))+O_{s_1}(t^3)|^\frac{3}{2}}
= \frac{3s_1}{4}\frac{1}{|t|}( 1 + O_{s_1}(t))$ and as $s_1 > 0$,
the curvature function goes to infinity when $t$ tends to $0$. As
the inflections of $\g_s$ are zeros of $\kp_s$ and vertices are
maximum or minimum points of $\kp_s$, then we can conclude that the
vertices are minimum points of $\kp_s$ (that is,
$\kp_s^{\prime\prime}>0$) and, consequently, we get
$\kp_s\kp_s^{\prime\prime} <0$ at the first vertex and
$\kp_s\kp_s^{\prime\prime} >0$ at the second.

Going around the origin in the $(s_1,s_2)$-space and using the configuration on
the cusp stratum and the results in sections $5$ and $7$ we get the
configuration of $\g_s$ in each stratum as shown in Figure
\ref{fig:def-cusp-ramph}.

(iii) It follows from (ii) and the results in sections $5$ and $7$.

(iv) From Theorem \ref{teo:est-cusp-ramp}, the stratifications
in the parameters spaces of $\g_s$ and $\a_u$ are homeomorphic, see
Figure \ref{fig:def-cusp-ramph} center. Moreover, the calculations in (ii) and (iii) depend only on the fact that the curve has a ramphoid cusp $\mathcal{A}_h$-equivalent to
$(t^2,t^4+t^5)$ and the family is \textit{FRS}-generic, conditions that $\a_u$ satisfies.
\end{proof}

%%%%%%%%%%%%%%%%%%%%%%%%%%%%%%%%%%%%%%%%%%%%%%%%%%%%%%%%%%%%%%%%%%%%%%%%%%%%%%%%%%%%%%%%%%%%%%%%%%%%%%%%%%%%%%%%%%%%%%%%
%%%%%%%%%%%%%%%%%%%%%%%%%%%%%%%%%%%%%%%%%%%%%%%%%%%%%%%%%%%%%%%%%%%%%%%%%%%%%%%%%%%%%%%%%%%%%%%%%%%%%%%%%%%%%%%%%%%%%%%%

%%%%%%%%%%%%%%%%%%%%%%%%%%%%%%%%%%%%%%%%%%%%%%%%%%%%%%%%%%%%%%%%%%%%%%%%%%%%%%%%%%%%%%%%%%%%%%%%%%%%%%%%%%%%%%%%%%%%%%%%
%%%%%%%%%%%%%%%%%%%%%%%%%%%%%%%%%%%%%%%%%%%%%%%%%%%%%%%%%%%%%%%%%%%%%%%%%%%%%%%%%%%%%%%%%%%%%%%%%%%%%%%%%%%%%%%%%%%%%%%%

%%%%%%%%%%%%%%%%%%%%%%%%%%%%%%%%%%%%%%%%%%%%%%%%%%%%%%%%%%%%%%%%%%%%%%%%%%%%%%%%%%%%%%%%%%%%%%%%%%%%%%%%%%%%%%%%%%%%%%%%
%%%%%%%%%%%%%%%%%%%%%%%%%%%%%%%%%%%%%%%%%%%%%%%%%%%%%%%%%%%%%%%%%%%%%%%%%%%%%%%%%%%%%%%%%%%%%%%%%%%%%%%%%%%%%%%%%%%%%%%%

\section*{Acknowledgements}

%We thank to the referees for their comments and suggestions which help us
%to improve the presentation of this paper.
%
%\smallskip

The author was supported by the FAPESP doctoral grant 2015/16177-2. 

The author would like to thank his supervisor Farid Tari for the great teachings, useful discussions and valuable comments.

\let\thefootnote\relax\footnotetext{\\ Alex Paulo Francisco\\
Universidade Federal do Paran{\'a}, Campus Pontal do Paran{\'a} - CEM\\
Av. Beira-Mar, s/n, caixa postal 61, Pontal do Sul, Pontal do Paran{\'a} - PR, Brazil\\
Email: alexpf@ufpr.br}

\end{document}